\documentclass[12pt]{amsart}
\usepackage{amssymb,amsmath}
\usepackage{hyperref}
\usepackage{amsthm,amsfonts}
\usepackage{footnote}
\usepackage{textcomp}
\usepackage[english]{babel}
\usepackage{graphicx, color}
\usepackage{xargs}                     % Use more than one optional parameter in a new commands
\usepackage[pdftex,dvipsnames]{xcolor}  % Coloured text etc.
\usepackage[colorinlistoftodos,prependcaption,textsize=tiny]{todonotes}
\newcommandx{\unsure}[2][1=]{\todo[linecolor=red,backgroundcolor=red!25,bordercolor=red,#1]{#2}}
\newcommandx{\change}[2][1=]{\todo[linecolor=blue,backgroundcolor=blue!25,bordercolor=blue,#1]{#2}}
\newcommandx{\info}[2][1=]{\todo[linecolor=OliveGreen,backgroundcolor=OliveGreen!25,bordercolor=OliveGreen,#1]{#2}}
\newcommandx{\improvement}[2][1=]{\todo[linecolor=Plum,backgroundcolor=Plum!25,bordercolor=Plum,#1]{#2}}
\newcommandx{\thiswillnotshow}[2][1=]{\todo[disable,#1]{#2}}
\newtheorem{theorem}{Theorem}[section]
\newtheorem{lemma}[theorem]{Lemma}
\newtheorem{proposition}[theorem]{Proposition}
\newtheorem{corollary}[theorem]{Corollary}

\newtheorem{question}[theorem]{Question}

\theoremstyle{definition}
\newtheorem{definition}[theorem]{Definition}

\newtheorem{remark}[theorem]{Remark}
\newtheorem{example}[theorem]{Example}

\newtheorem{notation}[theorem]{Notation}

\newcommand{\fm}{\mathfrak m}
\newcommand{\st}{\star}
\newcommand{\bx}{\mathbf x}
\newcommand{\bz}{\mathbf z}
\newcommand{\by}{\mathbf y}

\newcommand{\bu}{\mathbf u}

\newcommand{\bv}{\mathbf v}
 
\newcommand{\PP} {{\mathbb{P}}}

\newcommand{\CC} {{\mathbb{C}}}

\usepackage{relsize}
\begin{document}
\sloppy
%%%%%%%%%%%%%%%%%%%%%%%%%%%%%%%%%%%%%%%%%%%%%%%%%%%%%%%%%%%%%%%%%%%%%%%%

\title[Hadamard products of symbolic powers and Hadamard fat grids]{Hadamard products of symbolic powers and Hadamard fat grids}
\thanks{}

\author{I. Bahmani Jafarloo}
\address{Dipartimento di Matematica e Informatica, Università degli studi di Catania, Viale A. Doria, 6, 95100  Catania, Italy}
\email{iman.bahmanijafarloo@polito.it}

\author{C. Bocci}
\address{Department of Information Engineering and Mathematics,
Università degli studi di Siena, Via Roma 56, 53100 Siena, Italy}
\email{cristiano.bocci@unisi.it}
\author{E. Guardo}
\address{Dipartimento di Matematica e Informatica, Università degli studi di Catania, Viale A. Doria, 6, 95100  Catania, Italy}
\email{guardo@dmi.unict.it}
\author{G. Malara}
\address{Pedagogical University of Cracow, Institute of
	Mathematics, Podchor\c{a}\.{z}ych 2, PL-30-084 Kraków, Poland}
\email{grzegorzmalara@gmail.com}


\begin{abstract}
In this paper we address the question if, for points $P, Q \in \mathbb{P}^{2}$, $I(P)^{m} \star I(Q)^{n}=I(P \star Q)^{m+n-1}$ and we obtain different results according to the number of zero coordinates in $P$ and $Q$.
 Successively, we use our results to define the so called Hadamard fat grids, which are the result of the Hadamard product of two sets of collinear points with given multiplicities. The most important invariants of Hadamard fat grids,  as minimal resolution, Waldschmidt constant and resurgence, are then computed.
\end{abstract}

\maketitle

\section{Introduction}

In the last few years, the Hadamard products of projective varieties have been widely studied from the point of view of Projective Geometry and Tropical Geometry. The main problem in this setting is the behaviour of the Hadamard product between varieties with many points with zero coordinates.

The paper \cite{BCK}, where Hadamard products of general linear spaces are studied,  can be considered the first step in this direction. Successively, the second author, with Calussi, Fatabbi and Lorenzini, in \cite{BCFL1} and \cite{BCFL2}, address the Hadamard products of linear varieties not necessarily in general position.  Moreover, they show that the  Hadamard
product of two generic linear varieties $V$ and $W$ is projectively equivalent to a Segre embedding and  that the Hilbert function of the Hadamard product $V\star W$ of two  varieties, with $\dim(V), \dim(W)\leq 1$, is the product of the Hilbert functions of the original varieties $V$ and $W$.  

An important result contained in \cite{BCK} concerns the construction of star configurations of points, via Hadamard product. This construction found a generalization in \cite{CCGVT} where the authors introduce a new construction using the Hadamard product to present star configurations of codimension $c$ of $\PP^n$ and which they called Hadamard star configurations. Successively,  the first author and Calussi introduce a more general type of Hadamard star configuration; any star configuration constructed by their approach is called a weak Hadamard star configuration. In \cite{BJC} they classify weak Hadamard star configurations and, in the case $c=n$, they investigate the existence of a (weak) Hadamard star configuration which is apolar to the generic homogeneous polynomial of degree $d$. The use of Hadamard products in this context permits a complete control both in the coordinates of the points forming the star configuration and  the equations of the hyperplanes involved on it. This fact leads to the question if other interesting set of points can be constructed by Hadamard product. Recent results  in this direction can be found in \cite{BCC22}, where the second author along with  C. Capresi and D. Carrucoli  built Gorenstein set of points in $\PP^3$, with given $h-$vector, by creating, via Hadamard products, a stick figure of lines to which they apply the results of Migliore and Nagel \cite{MN}, and also in \cite{CCFG} where the third author along with E. Carlini, M.V. Catalisano and G. Favacchio  found a relation between star configurations where all the hyperplanes are osculating to the same rational normal curve (contact star configurations)  and Hadamard products of linear varieties.

This paper addresses another question, proposed by the second author during the CMO Workshop ``Ordinary and Symbolic Powers of Ideals'' (May 14-19, 2017, Oxaqua, Mexico) stating

\begin{question}\label{Q1}
	Is it true that for $P, Q$ points in $\mathbb{P}^{2}$, $I(P)^{m} \star I(Q)^{n}=I(P \star Q)^{m+n-1}$ ?
\end{question}

The answer of this question is given in Theorem \ref{thmtwopoints}. Successively this result is used in Section 5 to define the so called Hadamard fat grids, which are the result of the Hadamard product of two sets of collinear points with given multiplicities. Finally, we compute algebraic invariants of Hadamard fat grids, as minimal degree, minimal resolution, Waldschmidt  constant and resurgence.
We also point out that the results of Section \ref{HFG} enlarge the known literature on the minimal graded resolution of sets of fat points in $\mathbb P^2$ with all the same multiplicities  supported on a complete intersection (see for instance \cite{CooperGuardo, GVT1}). 

\vskip0.3cm
{\bf Acknowledgments.} I. Bahami Jafarloo thanks C. Bocci for hospitality and for providing scientific discussion during his staying in Siena. I. Bahami Jafarloo's work has been supported by Universit\`{a} degli Studi di Catania, grant  \lq\lq Hadamard product and symbolic powers”, D.R. n. 3158 , 17/09/2021,\lq\lq Progetto Piaceri, Linea di intervento 2" ,  E. Guardo's work has  been supported by the Universit\`{a} degli Studi di Catania, ``Progetto Piaceri, Linea di intervento 2". C. Bocci and E. Guardo are members of GNSAGA of INdAM (Italy).

The authors thank G. Favacchio for some useful discussions on a preliminary version of the paper.

\section{Preliminary results on join and Hadamard products of ideals}
Let  $S = \mathbb{K}[\mathbf{x}]=\mathbb{K}\left[x_{0}, \ldots, x_{N}\right]$  be a polynomial ring over an algebraically closed field. Let $\fm=\left\langle x_{0}, \ldots, x_{N}\right\rangle$ be the homogeneous irrelevant ideal. 

Let $I_1, I_2, \dots I_r$ be ideals in $S$. We introduce $(N+1)r$ new unknowns, grouped in $r$ vectors $\mathbf{y}_j=(y_{j0},\dots, y_{jN})$, $j=1,2,\dots, r$ and we consider the polynomial ring $\mathbb{K}[\mathbf{x},\mathbf{y}]$ in all $(N+1)r+N+1$ variables.

Let $I_j(\mathbf{y}_j)$ be the image of the ideal $I_j$ in $\mathbb{K}[\mathbf{x},\mathbf{y}]$ under the map $\mathbf{x} \mapsto \mathbf{y}_j$. Then the {\it join}  $I_1*I_2* \cdots *I_r$, as defined in \cite{SS06}, is the elimination ideal
\[
\left(I_1(\mathbf{y_1})+\cdots + I_r(\mathbf{y_r})+\left\langle x_{i}-y_{1i}-y_{2i}-\cdots-y_{ri} \mid i=0, \ldots, N\right\rangle\right) \bigcap \mathbb{K}[\mathbf{x}]
\]
and their Hadamard product $I_1\star I_2\star\cdots \star I_r$, as defined in \cite{BCK}, is the elimination ideal
\[
\left(I_1(\mathbf{y_1})+\cdots + I_r(\mathbf{y_r})+\left\langle x_{i}-y_{1i}y_{2i}\cdots y_{ri} \mid i=0, \ldots, N\right\rangle\right) \bigcap \mathbb{K}[\mathbf{x}].
\]

We define the $r$-th secant of an ideal $I \subset \mathbb{K}[\mathbf{x}]$ to be the $r$-fold join $I$ with itself:
$$
I^{* r}:=I * I * \cdots * I.
$$
Similarly we define the $r$-th Hadamard power of an ideal $I \subset \mathbb{K}[\mathbf{x}]$ to be the $r$-fold Hadamard product of $I$ with itself:
$$
I^{\star r}:=I \st I \st \cdots \st I.
$$

\begin{remark}
The Hadamard product of points in a projective space is a coordinate-wise product as in the case of the Hadamard product of matrices.
Let $p,q\in \PP^N$ be two points with coordinates $[p_0:p_1:\cdots:p_N]$ and $[q_0:q_1:\cdots:q_N]$ respectively. If $p_iq_i\not= 0$ for some $i$, the Hadamard product $p\star q$ of $p$ and $q$, is defined as $p\star q=[p_0q_0:p_1q_1:\cdots:p_Nq_N]$. If  $p_iq_i= 0$ for all $i=0,\dots, n$ then we say $p\star q$ is not defined.
This definition extends to the Hadamard product of varieties in the following way.
Let $X$ and $Y$ be two varieties in $\PP^N$. Then the Hadamard product $X\star Y$ is defined as
\[X\star Y=\overline{\{p\star q: p\in X, q\in Y, p\star q\mbox{ is defined} \}}.\]

Thus the defining ideal of the Hadamard product $X\star Y$ of two varieties $X$ and $Y$, that is, the ideal $I(X\star Y)$, equals  the Hadamard product  of the ideals $I(X)\star I(Y)$  (see \cite[Remark 2.6]{BCK}).

\end{remark}

\begin{definition}
If $I$ is a homogeneous ideal of $S$, the $m$-th symbolic power of $I$ is defined as $$I^{(m)}=\bigcap_{p\in \textrm{Ass}(I)}{(I^mS_{p}\cap S)}$$ \noindent where $Ass(I)$ denotes the set of the associated primes of $I$. If $I$ is a radical ideal then $I^{(m)}=\bigcap_{p\in \textrm{Ass}(I)}{p^m}.$

In this paper we will always deal with ideals of  fat points.  Given distinct points $p_1,\dots, p_s\in \mathbb P^N$ and nonnegative integers $m_i$ (not all 0), let $Z = m_1p_1 +\dots+m_sp_s$ denote the scheme (called a set of fat points) defined by the ideal $I_Z=\bigcap_{i=1}^s{I(P_i)^{m_i}}$, where  $I(P_i)$ is the ideal generated by all homogeneous polynomials  vanishing at $P_i$. When  all the $m_i$ are equal, we say that  $Z$ is a homogeneous set of fat points.
For ideals of this type, the $m$-th symbolic power can be simply defined as $I_Z^{(m)}=I_{mZ}=\bigcap_{i=1}^s{I(P_i)^{mm_i}}$. 
\end{definition}

If $I^m$ is the regular power of an ideal $I$, then there is clearly a containment $I^m\subseteq I^{(m)}$
and indeed, for $0\neq I\subsetneq R$, $I^r\subseteq I^{(m)}$ holds if and only if $r\ge m$
\cite[Lemma 8.1.4]{PSC}.
A much more difficult problem is to determine when there are containments
of the form $I^{(m)}\subseteq I^r$.
The results of \cite{ELS} and \cite{HH1} show that
$I^{(m)}\subseteq I^r$ holds whenever $m\ge Nr$, where $N$ is the codimension of the ideal.

As a stepping stone, the second author and B. Harbourne (\cite{BH}) introduce an asymptotic
quantity which we refer to as 
the {\it resurgence}, namely 
\[
\rho(I)=\hbox{sup}\{m/r : I^{(m)}\not\subseteq I^r\}.
\]
In particular, if $m > \rho(I)r$, then one is guaranteed
that $I^{(m)}\subseteq I^r$. 
There are still, however, very few cases for which the actual value of
$\rho(I)$ is known, and they are almost all cases for which
$\rho(I)=1$. For example, by Macaulay's unmixedness theorem
it follows that $\rho(I)=1$ when $I$ is a complete intersection (see Proposition \ref{CI}). And if $I$ is a monomial ideal, 
it is sometimes possible to compute $\rho(I)$ directly. In this paper we will show that if $I$ is the defining ideal of a Hadamard fat grid, $\rho(I)=1$ even if $I$ is a not complete intersection (see Proposition \ref{equality symb and reg} and Corollary \ref{brian}).

We mention few results which will be useful for the rest of the paper.

\begin{proposition}\label{SS}
 Let $I$ be a radical ideal in a polynomial ring over an algebraically closed field. Then
$I^{(t)}=I * \fm^{t}$.
\end{proposition}
\begin{proof}
See \cite[Proposition 2.8]{S08}.
\end{proof}
\begin{proposition}\label{CI}
Let $I$ be a complete intersection ideal in a polynomial ring over an algebraically closed field. Then
	$I^{(t)}=I^t$ for all $t\geq 1$. 
\end{proposition}
\begin{proof}
See	\cite[Lemma 5 and Theorem 2 of Appendix 6]{ZS75}.
\end{proof}

\begin{lemma}\label{interideal}
The Hadamard product distributes over intersections:
$$
\left(\bigcap_{l \in \mathcal{L}} J_{l}\right) \st K=\bigcap_{l \in \mathcal{L}}\left(J_{l} \st K\right)
.$$
\end{lemma} 
\begin{proof}
The proof is analogous to \cite[Lemma 2.6]{S08}. A polynomial $f$ belongs to $\left(\cap J_{l}\right) \st K$ if and only if $f(\mathbf{y}_1\mathbf{y}_2) \in\left(\cap J_{l}\right)(\mathbf{y}_1)+K(\mathbf{y}_2)$ if and only if $f(\mathbf{y}_1\mathbf{y}_2) \in J_{l}(\mathbf{y}_1)+K(\mathbf{y}_2)$ for all $l \in \mathcal{L}$ if and only if $f \in \cap\left(J_{l} \st K\right)$.
\end{proof}

From the previous lemma we get the following

\begin{corollary}
\label{coldechad}
Let $I,J$ be two ideals in $K[x_0,\dots,x_N]$ with primary decomposition respectively
$
I=I_1\cap I_2\cap \cdots \cap I_s$  and
$
J=J_1\cap J_2\cap \cdots \cap J_t
$,
then
\begin{equation}\label{dechad}
I \star J=\bigcap_{
\begin{array}{c}
1\leq i \leq s\\
1\leq j \leq t
\end{array}} I_i\star J_s.
\end{equation}
\end{corollary}

\begin{remark}\rm
The right-hand term in (\ref{dechad}) is not in general a minimal primary decomposition of $I\st J$ since it can contains some redundant term. As a matter of fact consider the ideals in $K[x_0,x_1,x_2]$
\[
\begin{array}{l}
I=\langle 5x_0-x_1-x_2, 6x_1^2-13x_1x_2+6x_2^2\rangle\\
J=\langle 5x_0-4x_1-3x_2, 16x_1^2-26x_1x_2+9x_2^2\rangle
\end{array}
\]
with primary decompositions
\[
\begin{array}{l}
I=I_1\cap I_2=\langle 3x_1-2x_2, x_0-\frac15 x_1-\frac15 x_2 \rangle \cap \langle 2x_1-3x_2, x_0-\frac15x_1-\frac15x_2\rangle\\
J=J_1\cap J_2=\langle 2x_1-x_2, x_0-\frac45 x_1-\frac35 x_2\rangle\cap \langle 8x_1-9x_2,x_0-\frac45 x_1-\frac35 x_2\rangle.
\end{array}
\]
The primary decomposition of $I\st J$ is 
\[
\langle 16x_1-27x_2, 4x_0-3x_2\rangle \cap \langle 4x_1-3x_2, 2x_0-x_2\rangle \cap \langle 3x_1-x_2, 3x_0-x_2\rangle
\]
which does not have four ideals as expected. This is due to the fact that
$I_2\st J_1=I_1\st J_2$.

\end{remark}
\begin{lemma}\label{mink_join}
One has	$  \fm ^m * \fm ^n=\fm^{m+n-1}$.
\end{lemma}
\begin{proof}
Let  $\bu=\left(u_{0}, \ldots, u_{N}\right)$ an integer vector and denote by
 $$
\fm^{\mathbf{u}}=\left\langle x_{i}^{u_{i}}: u_{i}> 0\right\rangle
.$$

	We decompose $  \fm ^m = \cap_{\bu} \fm^\bu$ and $  \fm^n = \cap_{\bv} \fm^\bv$. Then  the proof follows from \cite[Lemma 2.3] {SS06} and \cite[Lemma 2.6]{S08}.
\end{proof}

\begin{definition}
Let $H_i\subset\PP^n,i=0,\ldots,n$, be the hyperplane $x_i=0$ and set
$$\Delta_i=\bigcup_{0\leq j_1<\ldots<j_{n-i}\leq n}H_{j_1}\cap\ldots\cap H_{j_{n-i}}.$$
\end{definition}

In other words, $\Delta_i$ is the $i-$dimensional variety of points having {\em at most} $i+1$ non-zero coordinates. Thus $\Delta_0$ is the set of coordinates points and $\Delta_{n-1}$ is the union of the coordinate hyperplanes. Note that elements of $\Delta_i$ have {\em at least} $n-i$ zero coordinates. We have the following chain of inclusions:
\begin{equation}\label{inclusiondelta}
\Delta_0=\{[1:0:\cdots:0],\ldots,[0:\cdots:0:1]\}\subset\Delta_1\subset\ldots\subset\Delta_{n-1}\subset\Delta_{n}=\PP^n.\end{equation}

Let $R$ be the ring $\mathbb{K}[x_0, x_1, \dots, x_n]$. Given a vector of nonnegative integers $I=(i_0, \dots, i_n)$, we denote by $X^I$ the monomial $x_0^{i_0}x_1^{i_1}\cdots x_n^{i_n}$ and by $|I|=i_0+\cdots +i_n$. Similarly, if $P$ is a point of $\PP^n$ with coordinates $[p_0:p_1:\cdots :p_n]$, we denote by $P^I$ the monomial $X^I$ evaluated in $P$, that is $p_0^{i_0}p_1^{i_1}\cdots p_n^{i_n}$.
Moreover, if $P$ is a point of $\PP^n\setminus \Delta_{n-1}$ with coordinates $[p_0:p_1:\cdots :p_n]$, we denote by $\frac{1}{P}$ the point with coordinates $[\frac{1}{p_0}:\frac{1}{p_1}:\cdots :\frac{1}{p_n}]$.

\begin{definition}\label{HADT}
Let $f\in \CC[x_0, x_1, \dots, x_n]$ be a homogenous polynomial, of degree $d$, of the form $f=\sum_{|I|=d}a_IX^I$ and consider a point $P\in \PP^n\setminus \Delta_{n-1}$. The Hadamard transformation of $f$ by $P$ is the polynomial
\begin{equation}\label{Hadtransformation}
f^{\star P}=\sum_{|I|=d}\frac{a_I}{P^I}X^I.
\end{equation}
\end{definition}

\begin{theorem}\label{ThmGens}
Let $I$ be an ideal in $\CC[x_0, \dots, x_n]$ and consider a point $P\in \PP^n\setminus \Delta_{n-1}$. If $f_1, \dots, f_s\subset \CC[x_0, \dots, x_n]$ is a generating set for $I$, that is $I=\langle f_1, \dots, f_s\rangle$, then $f_1^{\star P}, \dots, f_s^{\star P}$ is a generating set for $I(P) \star I$.

Moreover, if $f_1, \dots, f_s$ is a Gr\"obner basis for $I$, then $f_1^{\star P}, \dots, f_s^{\star P}$ is a Gr\"obner basis for $I(P) \star I$.
\end{theorem}
\begin{proof}
See \cite[Theorem 3.5]{BC}.
\end{proof}

\section{Preliminary results on ACM sets of fat points in $\mathbb P^1\times \mathbb P^1$}\label{p1xp1}

In this section we recall some known results and a standard technique used for  sets of fat points in $\mathbb P^1\times \mathbb P^1$ since they are the main tools to find a minimal graded free resolution of  special sets of fat points in $\mathbb P^2$ called Hadamard fat grids (see Definition \ref{defHFG} in Section \ref{HFG}). Indeed,  we prove that a Hadamard fat grid inherits some properties from an arithmetically Cohen-Macaulay set of a fat points in $\mathbb P^1\times \mathbb P^1$, such as that its defining ideal is generated by product of linear forms. We should refer the reader to Chapters 4 and 6 in \cite{GVT} for more details on arithmetically Cohen-Macaulay sets of fat points in $\mathbb P^1\times \mathbb P^1$.

Let $Z$ be a finite set of points in  $\mathbb P^1\times\mathbb P^1$, and let $R/I_Z$
denote the associated $\mathbb N^2$-graded coordinate ring. When $R/I_Z$ is Cohen-Macaulay, that is depth $R/I_Z = dim R/I_Z= 2$, then $Z$ is called
an arithmetically Cohen-Macaulay (ACM for short) set of points.

\begin{remark}\label{FGM} For the ease of the reader, we now recall  a standard argument that relates sets of (fat) points in $\mathbb P^1\times\dots \times\mathbb P^1$  ($n$-times) and sets of (fat) points in $\mathbb P^{2n-1}$ in the case $n=2$ (see for more details Section 3,  Theorem 3.2, Corollaries 3.3  and 3.4 in \cite{FGM} and, for instance, also \cite{FM1, FM2, GHVT}).  We observe that in our case in $\mathbb{P}^1$ a point is also a hyperplane, and this allows us to use hyperplane sections and  related constructions for our study. 

Let $R = \mathbb K[x_{1,0},x_{1,1}, x_{2,0},x_{2,1}]=\mathbb K[\mathbb P^1\times \mathbb P^1]$ be the coordinate ring for $\mathbb P^1\times \mathbb P^1$, which we shall also view as the coordinate ring for $\mathbb P^{3}$. Let $Z \subset \mathbb P^1\times \mathbb P^1$ be a finite set of fat points. Since $I_Z$ defines both a set of fat points in $\mathbb P^1\times \mathbb P^1$ and a union of linear varieties (fat lines) in $\mathbb P^{3}$, by abuse of notation, we denote by $Z$ also the subvariety of $\mathbb P^{3}$ defined by this ideal. 

In  $\mathbb P^1\times \mathbb P^1$ a complete intersection  of type $(r,0)$ and $(0,s)$, or simply an $(r,s)$-grid, can be viewed as $2$ families of linear forms, and we denote by $A_{1,i}$ the linear combinations of $x_{1,0}$ and $x_{1,1}$, and by $A_{2,i}$ the linear combinations of $x_{2,0}$ and $x_{2,1}$. Set\[
\begin{array}{c}
|\{ A_{1,i} \ | \ A_{1,i} \cap X \neq \emptyset \}| =r\\
|\{ A_{2,i}  \ | \ A_{2,i} \cap X \neq \emptyset \}| = s.
\end{array}
\]

Let $T$ be the polynomial ring in $r+s$ variables $a_{1,1},\dots, a_{1,r}, a_{2,1},\dots, a_{2,s}$. We form the monomial ideal in $T$ given by the intersection of ideals of the form $(a_{1,i}, a_{2,j})^{m_{ij}}$ corresponding to the components of $X$. This intersection defines a height $2$ monomial ideal, $J\subset T$.  Following Theorem 2 in \cite{FGM},  a set of fat points $Z$ in $\mathbb P^1\times \mathbb P^1$  is ACM if and only if $J\subset T$ is CM, where $T$ is  the ring previously defined. Since $Z$ can be viewed as an ACM set of lines in $\mathbb P^3$ (it is $1$-dimensional) we can still find a suitable hyperplane section in order to get a set $X$ of points in $\mathbb P^2$ that shares the same Betti numbers as $Z$. In particular, with the above described method, we will show that a Hadamard fat grid in $\mathbb P^2$ share the same graded Betti numbers as  an ACM set of fat points in $\mathbb P^1\times \mathbb P^1$ (see Theorem \ref{samebetti}).

\end{remark}

In the sequel, it  useful to consider in $\mathbb Z\times \mathbb Z$  and in
$\mathbb N \times \dots \times \mathbb N$ the partial (lexicographic) ordering induced by the usual one in  $\mathbb Z$ and in $\mathbb N$, respectively; we will denote it by “$\leq$”.

Consider an $(r,s)$-grid, and let $M=\{m_1,\ldots,m_r\}$, $ N=\{n_1,\ldots,n_s\}$ be two sets of nonnegative integers with $m_1\leq m_2\leq \dots\leq m_r$ and $n_1\leq n_2\leq \dots\leq n_s$. We have

\begin{align*}\label{condhfg}
\max_{i,j}\{m_i+n_j-1\}&=m_r+n_s-1,\\
\max_{i}\{m_i+n_j-1\}&=m_r+n_j-1,\\
\max_{j}\{m_i+n_j-1\}&=m_i+n_s-1.
\end{align*}

For each $i=1,\dots,r$ and $j=1,\dots,s$, set
\begin{align*}
t_{ij}(h)=(m_{i}+n_j-1-h)_{+}=\max\{0,m_{i}+n_j-1-h\} \textrm{ for $h\in \mathbb N^0$.}
\end{align*}

For $i=1\dots, r$, let $\mathcal S$ be the set of the $s$-tuples of type 

\begin{equation}\label{S}
\mathcal S=\left\{(t_{i1}(h), t_{i2}(h),\ldots,t_{is}(h))\right\}.
\end{equation}

The next Lemma \ref{sto}  shows that the set  $\mathcal S$ of tuples associated to $M$ and $N$ is totally ordered. It will be useful to find the graded Betti numbers of a Hadamard fat grid (see Section \ref{HFG}).

We have
\begin{lemma}\label{sto} For $i=1\dots, r$ and for  $h=0\dots, m_1+n_s-2$, the set $\mathcal S$ is a totally ordered set of tuples with usual lex ordering.
\end{lemma}
\begin{proof} Let $r=1$. From construction, it is $\max_{j}\{m_1+n_j-1\}=m_1+n_s-1$. For $h=0\dots, m_1+n_s-2$, we have
{\small{\begin{align*}
\mathcal S=&\{(t_{11}(h), t_{12}(h),\ldots,t_{1s}(h))\}=\\
&\{ ((m_1+n_1-1-h)_+,(m_1+n_2-1-h)_+,\dots, (m_1+n_s-1-h)_+,)\}=\\ 
&\{ (m_1+n_1-1, m_1+n_2-1,\dots, m_1+n_s-1), (m_1+n_1-2, m_1+n_2-2,\dots, m_1+n_s-2), \\
&\dots, 
(m_1+n_1-1-m_1-n_s+2,\dots, m_1+n_s-1- m_1-n_s+2)\}=\\
&\{ (m_1+n_1-1, m_1+n_2-1,\dots, m_1+n_s-1), (m_1+n_1-2, m_1+n_2-2,\dots, m_1+n_s-2),\\
& \dots, 
((n_1-n_s+1)_+, (n_2-n_s+1)_+\dots, (1))\}.
\end{align*}}}

Since $n_1\leq n_2\leq \dots\leq n_s$,  it is $m_1+n_{j_1}-1\leq m_1+n_{j_2}-1$ for all $j_1\neq j_2$, and since $n_1\leq n_s$ we have $n_1-n_s+1\leq 0$ if $n_1<n_s$ and $n_1-n_s+1=1$ if $n_1=n_s.$ Hence $\mathcal S$ is totally ordered for $r=1.$

Suppose $r>1$ and the lemma true for all the sets $\mathcal S' $ of $s$-tuples with $r'<r$ and prove it for $\mathcal S$. Let $1\leq k\leq r$ an integer such that $m_k+n_{j_1}-1\leq m_k+n_{j_2}-1$ for all $j_1\neq j_2$. From our construction, we have that $k=r$.

Define for $i=1\dots, r$
\begin{align*}
\mathcal S'=&\{(t'_{i1}(h), t'_{i2}(h),\ldots, t'_{is}(h))\}
\end{align*}
\noindent where 
\begin{align*}
t'_{ij}(h)=
\begin{cases}
(m_{i}+n_j-1-h)_{+} & \text{ if  $i\neq k$}\\
(m_{k}+n_j-2-h)_{+}& \text{ if $i=k$}.
\end{cases}
\end{align*}

By induction $\mathcal S'$ is totally ordered. We observe that 
$$\mathcal S=\mathcal S' \cup \{ ((m_k+n_1-1-h)_+,(m_k+n_2-1-h)_+,\dots, (m_k+n_s-1-h)_+)\}$$
\noindent and by construction and for all $i=1\dots,r=k$, we have

$$((m_i+n_1-1-h)_+,\dots, (m_i+n_s-1-h)_+)\leq ((m_k+n_1-1-h)_+,\dots, (m_k+n_s-1-h)_+).$$
Thus the set $\mathcal S$ is totally ordered.
\end{proof}

For each integer $0 \leq h \leq m_i+n_s-2$, with $i=1,\dots,r$  we define
\[
a_{i,h} := \sum_{e=1}^{s} (m_{i}+n_e-1 - h)_+ \hspace{.5cm}
\mbox{where $(n)_+ := \max\{n,0\}$.}
\]
Define
\begin{eqnarray*}
\tilde{\mathcal A}_Z & :=&
(a_{1,0},\ldots,a_{1,m_1+n_s-2},a_{2,0},\ldots,a_{2,m_2+n_s-2},\ldots,
a_{r,0},\ldots,a_{r,m_r+n_s-2}).
\end{eqnarray*}
Finally, we define $\mathcal A_Z$ to be the set of $(m_1+n_s-1+\dots+m_r+n_s-1)$-tuples,  that is, the set of  $(\sum_{i=1}^{r} m_i+rn_s-r)$-tuples that
one gets by rearranging the elements of $\tilde{\mathcal A}_Z$ in
descending order.

We now recall how to compute the graded Betti numbers of $I(Z)$ when
$Z$ is an ACM set of fat points in $\mathbb P^1\times \mathbb P^1$.   Let
$\mathcal A_Z=(\alpha_1,\ldots,\alpha_m)$ be the tuple associated to $Z$.
Define the following two sets from $\mathcal A_Z$:
\begin{eqnarray*}
C_{Z} & := & \left\{(m,0),(0,\alpha_1)\right\} \cup
\left\{(i-1,\alpha_i) ~|~ \alpha_i - \alpha_{i-1} < 0\right\} \\
V_{Z} & := & \left\{ (m,\alpha_m) \right\} \cup \left\{
(i-1,\alpha_{i-1}) ~|~ \alpha_i-\alpha_{i-1} < 0 \right\}
\end{eqnarray*}
\noindent where we take $\alpha_{-1} = 0$.  With this notation, we get that a minimal bigraded free resolution of an ACM set of points in $\mathbb P^1\times \mathbb P^1$ is given by 

\begin{theorem}\cite[Theorem 6.27]{GVT} \label{bettinumbersftpts}
Suppose that $Z$ is an ACM set of fat points in $\mathbb P^1\times \mathbb P^1$ with
$\mathcal A_Z = (\alpha_1,\ldots,\alpha_m)$.  Let $C_{Z}$ and $V_{Z}$ be
constructed from $\mathcal A_Z$ as above.  Then a bigraded minimal free
resolution of $I(Z)$  is given by
\[
0 \longrightarrow \bigoplus_{(v_1,v_2) \in V_{Z}} R(-v_1,-v_2)
\longrightarrow \bigoplus_{(c_1,c_2) \in C_{Z}} R(-c_1,-c_2)
\longrightarrow I(Z) \longrightarrow 0.\]
\end{theorem}

The following result will be one of the main tool that allows us to find that the Betti numbers of a given  Hadamard fat grid in $\mathbb P^2$ whose support  is a complete intersection. 

\begin{proposition}\label{acm}  Let $Z$ be a set of fat points in $\mathbb P^1\times \mathbb P^1$ whose support is a complete intersection of type $(r,s)$ (or $(r,s)$-grid))  and whose multiplicities are of type $m_i+n_j-1$ for $i=1,\dots,r$ and $j=1,\dots,s$. Then $Z$ is ACM.
\end{proposition}
\begin{proof} From Theorem 6.21 in \cite{GVT} and Lemma \ref{sto}, we get the conclusion.
\end{proof}

\section{Hadamard product of symbolic powers of ideals of points}

We focus now our attention on the Hadamard product $I(P)^{m} \star I(Q)^{n}$, where $P,Q\in \PP^2$.

Hence, we work on the polynomial ring  $S = \mathbb{K}[\mathbf{x}]=\mathbb{K}\left[x_{0}, x_1, x_{2}\right]$, over an algebraically closed field and we still denote by  $\fm=\left\langle x_{0}, x_1, x_2 \right \rangle$ the irrelevant ideal. 
Since we will multiply, by Hadamard, only two ideals, we change the name of the variables involved in the process, with respect to our original definition. More precisely we will use the extra variables ${\mathbf y}=(y_0,y_1,y_2)$ and ${\mathbf z}=(z_0,z_1,z_2)$  for the Hadamard product of the ideals $I$ and $J$:
$$I \star J=\left(I(\mathbf{y})+J(\mathbf{z})+\left\langle x_{i}-y_{i}z_{i} \mid i=0,1,2\right\rangle\right) \bigcap \mathbb{K}[\mathbf{x}]
.$$
We set $\mathcal{H}= (x_0-y_0z_0,x_1-y_1z_1,x_2-y_2z_2) $.

We start with some preliminary results.

\begin{lemma}\label{gpnew}
	Let $ P=[p_0:p_1:p_2]$ be a point, then
	\begin{itemize}
	\item[i)] if $P\in \PP^2 \setminus \Delta_1$, then $I(P)\star\fm^t  =\fm^t$ ,
	\item[ii)] if $P\in \Delta_1$ then $ \fm^t \subseteq I(P)\star\fm^t $.
	\end{itemize}
\end{lemma}
\begin{proof}
For i) we can use Theorem \ref{ThmGens}. Hence the generators of $I(P)\star\fm^t$ are the Hadamard transformations, with respect to $P$, of the monomials in $\fm^t$.  According to Definition \ref{HADT}, such transformations of the monomials in  $\fm^t$ are the same monomials scaled by a constant $P^I$, hence $I(P)\star\fm^t=\fm^t$.

For ii) we can assume, without loss of generality,  that $ p_0=0$
\begin{flalign*}
		I(P) \st \fm^t &=  \left( I(P)(\mathbf{y})+\fm(\mathbf z)^n + \left( x_{i}-y_{i}z_{i} \mid i=0,1, 2\right)\right) \bigcap \mathbb{K}[\mathbf{x}]\nonumber\\
		& = \left( (y_0,p_2y_1-p_1y_2)+ \left( z_0,z_1, z_2\right)^t+  \left( x_{i}-y_{i}z_{i} \mid i=0, 1, 2\right)\right) \bigcap \mathbb{K}[\mathbf{x}]\\
		& = (x_0) + (x_1,x_2)^t
\end{flalign*}
where the last equality follows from the fact that $y_0$ annihilates $y_0z_0$. Hence  $\fm^t \subseteq I(P)\star\fm^t $.
Similarly, if the point $P$ has $p_0=p_1=0$ one has
	\begin{flalign*}
		I(P) \st \fm^t &=  \left( I(P)(\mathbf{y})+\fm(\mathbf z)^n + \left( x_{i}-y_{i}z_{i} \mid i=0,1, 2\right)\right) \bigcap \mathbb{K}[\mathbf{x}]\nonumber\\
		& = \left( (y_0,y_1)+ \left( z_0,z_1, z_2\right)^t+  \left( x_{i}-y_{i}z_{i} \mid i=0, 1, 2\right)\right) \bigcap \mathbb{K}[\mathbf{x}]\\
		& = (x_0,x_1) + (x_2^t)\supset \fm^t.
	\end{flalign*}
\end{proof}

The following result gives a positive answers to Question \ref{Q1} when the points $P$ and $Q$ have non-zero coordinates. 
\begin{theorem}\label{thmtwopoints}
	Let $ P $ and $ Q $ be two points in $\PP^2\setminus\Delta_{1}$. Then for $ m,n\geq 1$ one has $I(P)^{m} \star I(Q)^{n}=I(P \star Q)^{m+n-1}$.
	
\end{theorem}
\begin{proof}
From Proposition \ref{SS} we know that
\begin{equation}\label{ISt2p}
I(P)^m= I(P)*\fm^m \mbox{ and } I(Q)^n= I(Q)*\fm^n.
\end{equation}
Thus, $I(P)^m\star I(Q)^{n} =\left[ I(P)*\fm^m\right] \st \left[ I(Q)*\fm^n\right] $ can be computed by applying successively  the definition of  join of two ideals and of Hadamard products of ideals.

For this aim, we define 
\begin{eqnarray*}
\mathcal{H}_x&=& (x_0-x'_0x''_0,x_1-x'_1x''_1,x_2-x'_2x''_2), \\
\mathcal{H}_y&=&(x'_0-y'_0-y''_0,x'_1-y'_1-y''_1,x'_2-y'_2-y''_2) ,\\
\mathcal{H}_z&=&(x''_0-z'_0-z''_0,x''_1-z'_1-z''_1,x''_2-z'_2-z''_2) \\
\hat{\mathcal{H}}&=& \big( x_0-(y'_0+y''_0)(z'_0+z''_0) ,x_1-(y'_1+y''_1)(z'_1+z''_1) ,x_2-(y'_2+y''_2)(z'_2+z''_2) \big)  .
\end{eqnarray*}

Hence we have the following sequence of equalities
\begin{flalign*}
	I(P)^m\star I(Q)^{n} =&\left[ I(P)*\fm^m\right] \st \left[ I(Q)*\fm^n\right] &\\
	=& \left(  \left[ I(P)*\fm^m\right](\bx') + \left[ I(Q)*\fm^n\right](\bx'')+\mathcal{H}_x\right)\bigcap \mathbb{K}[\mathbf{x}] & \\
	=&\left( \left[  \left( I(P)(\by')+\fm^m(\by'')+\mathcal{H}_y\right)\bigcap \mathbb{K}[\mathbf{x}'] \right]\right. &\\ 
	&\qquad\qquad\left. + \left[ \left( I(Q)(\bz')+\fm^n(\bz'')+\mathcal{H}_z\right)\bigcap \mathbb{K}[\mathbf{x''}]  \right]+\mathcal{H}\right) \bigcap \mathbb{K}[\mathbf{x}]  &\\
	=&\left( I(P)(\by')+\fm^m(\by'') +  I(Q)(\bz')+\fm^n(\bz'')+\hat{\mathcal{H}}\right) \bigcap \mathbb{K}[\mathbf{x}]&\\
	=&\left( I(P)(\by')+I(Q)(\bz')+\fm^m(\by'') +\fm^n(\bz'')+\hat{\mathcal{H}}\right) \bigcap \mathbb{K}[\mathbf{x}].&
\end{flalign*}

Denote by  $ \widetilde{I} $ the ideal in the last equality in the previous formula, i.e. 
\[
 \widetilde{I} = I(P)(\by')+I(Q)(\bz')+\fm^m(\by'') +\fm^n(\bz'')+\hat{\mathcal{H}}.
 \]

Since the elements in $\hat{\mathcal{H}}$ can be seen as
\begin{equation}\label{H}
	x_i-(y'_i+y''_i)(z'_i+z''_i) =x_i - y'_iz'_i-y'_iz''_i-z'_iy_i''-y''_iz''_i=0
\end{equation}

we can substitute the ideal  $\hat{\mathcal{H}}$ with the ideal generated by
\[
\begin{array}{l}
k^{(1)}_i-y'_iz'_i=0,\\
k^{(2)}_i-y'_iz''_i=0,\\
k^{(3)}_i-z'_iy_i''=0,\\
k^{(4)}_i-y''_iz''_i=0,\\
x_i-k^{(1)}_i-k^{(2)}_i-k^{(3)}_i-k^{(4)}_i =0
\end{array}
\]
with $i=0,1,2$.

By definition,  $ k^{(1)}_i-y'_iz'_i = 0 $ generates $ I(P)\st I(Q)  = I(P\st Q)$. Note that $ k^{(2)}_i-y'_iz''_i = 0 $ generates $ I(P)\st \fm^n $. By Lemma \ref{gpnew} (i), we know that $ I(P)\st \fm^n  = \fm^n $. Therefore we have $ \tilde{k}^{(2)}_i-z''_i = 0 $. Similarly $ k^{(3)}_i-z'_iy''_i = 0 $ generates $ I(Q)\st \fm^m $. Using the same lemma we have that $ I(Q)\st \fm^m  = \fm^m $, hence $ \tilde{k}^{(3)}_i-y''_i = 0 $. Now we have 
\begin{equation}\label{newX}
	 x_i-k^{(1)}_i-\tilde{k}^{(2)}_i-\tilde{k}^{(3)}_i-k^{(4)}_i =0.
\end{equation}

A simple calculation shows that any term of a form $ f $ of degree $ d $ in $ \tilde{I} $ containing $ \prod\left( \hat{k}^{(2)}_i\right)^{r_i}  $ is zero if $ \prod\left( \hat{k}^{(2)}_i\right)^{r_i} \in \fm^n(\bz'') $ which follows that any term containing $\prod\left({y_i''}^{r_i}{z_i''}^{t_i} \right)$ is zero as well. A similar calculation for terms containing $ \prod\left( \hat{k}^{(3)}_i\right)^{t_i}\in\fm^m(\by'')  $ follows that the term containing $\prod\left({y_i''}^{r_i}{z_i''}^{t_i} \right)$ is zero. Therefore we conclude that any term of a form $ f $ of degree $ d $ in $ \tilde{I} $ containing $\prod \left( k^{(4)}_i\right)^{d_i}   $ is zero. Hence we can simply cancel $ k^{(4)}_i $ from (\ref{newX}). We have that $ x_i-k^{(1)}_i-\tilde{k}^{(2)}_i-\tilde{k}^{(3)}_i = 0 $. Let
\begin{eqnarray*}
\widetilde{\mathcal{H}}&=& \left( x_0-k^{(1)}_0-\tilde{k}^{(2)}_0-\tilde{k}^{(3)}_0,  x_1-k^{(1)}_1-\tilde{k}^{(2)}_1-\tilde{k}^{(3)}_1,  x_2-k^{(1)}_2-\tilde{k}^{(2)}_2-\tilde{k}^{(3)}_2  \right) \\
\widetilde{\mathcal{H}}_{x'}&=& \left( x'_0-k^{(1)}_0,  x'_1-k^{(1)}_1,  x'_2-k^{(1)}_2  \right) \\
\widetilde{\mathcal{H}}_{x''}&=& \left( x''_0-\tilde{k}^{(2)}_0-\tilde{k}^{(3)}_0,  x''_1-\tilde{k}^{(2)}_1-\tilde{k}^{(3)}_1,  x''_2-\tilde{k}^{(2)}_2-\tilde{k}^{(3)}_2  \right)\\
\widetilde{\mathcal{H}}_x&=& \left(  x_0-x''_0-x'_0,   x_1-x''_1-x'_1,  x_2-x''_2-x'_2  \right) . 
\end{eqnarray*}
 Therefore,  
\begin{flalign*}
 \widetilde{I}\cap \mathbb{K}[\mathbf{x}] &= \left( I(P)(\by')+I(Q)(\bz')+\fm^m(\by'') +\fm^n(\bz'')+ \widetilde{\mathcal{H}} \right)\bigcap \mathbb{K}[\mathbf{x}]\\
 &=\left( \left[  \left( I(P)(\by')+I(Q)(\bz')+\widetilde{\mathcal{H}}_{x'}\right)\bigcap \mathbb{K}[\mathbf{x}'] \right]\right. &\\ 
 &\qquad\qquad\left. + \left[ \left(\fm^m(\by'') +\fm^n(\bz'')+\widetilde{\mathcal{H}}_{x''}\right)\bigcap \mathbb{K}[\mathbf{x''}]  \right]+\widetilde{\mathcal{H}}\right) \bigcap \mathbb{K}[\mathbf{x}]  &\\
 &= \left(  \left[ I(P)\st I(Q)\right](\bx') + \left[ \fm^m *\fm^n\right](\bx'')+\widetilde{\mathcal{H}}_x\right)\bigcap \mathbb{K}[\mathbf{x}] & \\
 &=  \left[ I(P)\st I(Q)\right] * \fm^{m+n-1} & \\
 &=  I(P\st Q) * \fm^{m+n-1} & 
\end{flalign*}
and, by Proposition \ref{SS}, the last ideal is equal to  $I(P\st Q) ^{m+n-1}$ and the theorem is proved.
\end{proof}

\begin{remark}
If we remove the condition $P,Q\in \PP^2\setminus\Delta_{1}$ in Theorem \ref{thmtwopoints} we are able to proof only the following inclusion
$$ I(P)^{m} \star I(Q)^{n} \supset I(P \star Q)^{m+n-1}.$$
To prove this inequality is enough to apply the proof of Theorem \ref{thmtwopoints} using part ii) of Lemma \ref{gpnew}.
\end{remark}

We observe that  if $ P\star Q $ is defined, and $ m=n=1 $ then  it follows from the definition of Hadamard product that $ I(P) \star I(Q)=I(P \star Q)$, and hence Question \ref{Q1} has an affirmative answer.
Theorem \ref{thmtwopoints}  shows that Question \ref{Q1} has an affirmative answer, when we consider points which are not in the coordinates lines. 
When one of the point is taken in a coordinate line, the formula in Question \ref{Q1} is no more valid, as stated in the following:

\begin{proposition}\label{thmtwopoints2}
	Let $ P $ and $ Q $ be two points in $ \PP^2$. 
\begin{itemize}
	\item[(a)] If $ P \in \PP^2\setminus\Delta_{1}$ and  $ P\star Q $ is defined then for $ m=1 $ and $ n\geq 1 $ Question \ref{Q1} has an affirmative answer.
	\item[(b)] If $P, Q \in \Delta_{1} \backslash \Delta_{0}$ then Question \ref{Q1} has no affirmative answer but for $ m=n=1 $. 
\end{itemize}
\end{proposition}

\begin{proof}

\textbf{(a):} without loss of generality let $ P=[p_0:p_1:p_2] $ with $ p_i\neq 0 $ and assume that $ Q\in\Delta_{0} $ has exactly one non-zero coordinate, that is $ Q= [q_0:0:0] $. Assume that $ m=1 $. Since $ P\st Q = Q$, therefore we only need to show that $ 	I(P) \star I(Q)^{n} =I(Q)^{n}$. We have that
\begin{flalign*}
	I(P) \star I(Q)^{n} =& (p_2x_1-p_1x_2,p_2x_0-p_0x_2) \st (x_1,x_2)^n\\
	=&\left( (p_2y_1-p_1y_2,p_2y_0-p_2y_2) +(z_1,z_2)^n + \mathcal{H} \right) \bigcap \mathbb{K}[\mathbf{x}]\\
	=& (x_1,x_2)^n = I(Q)^n.
\end{flalign*}

Let $ m>1 $. Since  $ x_1^n\notin  I(Q)^{n+m-1} $ and we know that $ x_1^n\in  I(P)^m\st I(Q)^{n} $ therefore $ I(Q)^{n+m-1}\neq I(P)^m \star I(Q)^{n}$.

Note that in general we have
\begin{flalign*}
	I(P)^m \star I(Q)^{n} =& (p_2x_1-p_1x_2,p_2x_0-p_0x_2)^m \st (x_1,x_2)^n\\
	=&\left( (p_2y_1-p_1y_2,p_2y_0-p_2y_2)^m +(z_1,z_2)^n + \mathcal{H} \right) \bigcap \mathbb{K}[\mathbf{x}]\\
	=& (x_1,x_2)^n = I(Q)^n=I(P\st Q)^n.
\end{flalign*}

Now let $ Q= [q_0:q_1:0]\in \Delta_{1} \backslash \Delta_{0}  $ and $ m>1 $. We have that 
\begin{flalign*}
	I(P)^m \star I(Q)^{n} =&  (p_2x_1-p_1x_2,p_2x_0-p_0x_2)^m \st (q_1x_0-q_0x_1,x_2)^n\\
	=&\left( (p_2y_1-p_1y_2,p_2y_0-p_0y_2)^m +(q_1z_0-q_0z_1,z_2)^n + \mathcal{H} \right) \bigcap \mathbb{K}[\mathbf{x}]\\
	=& (x_2^n)+(\cdots).
\end{flalign*}
Since $ x_2^n\notin I(P\st Q)^{m+n-1} $ therefore we have that $  I(P\st Q)^{m+n-1} \neq  	I(P)^m \star I(Q)^{n}$.

If $ m=1 $ we have:
\begin{flalign*}
	I(P)\star I(Q)^{n} =&  (p_2x_1-p_1x_2,p_2x_0-p_0x_2) \st (q_1x_0-q_0x_1,x_2)^n\\
	=&\left( (p_2y_1-p_1y_2,p_2y_0-p_0y_2)+(q_1z_0-q_0z_1,z_2)^n + \mathcal{H}
	\right) \bigcap \mathbb{K}[\mathbf{x}]\\
	=& (x_2^n)+(x_2^{n-1}(p_1q_1x_0-p_0q_0x_1))+\cdots+(p_1q_1x_0-p_0q_0x_1)^n\\
	=&(x_2,(p_1q_1x_0-p_0q_0x_1))^n=I(P\st Q)^n.
\end{flalign*}

\textbf{(b):} Let $ m\leq n $. 
Without loos of generality assume that $ P=[p_0:0:p_2] $  and $ Q=[q_0:q_1:0] $.  We have that $ I(P\st Q)=(x_1,x_2)$.
\begin{flalign*}
	I(P)^m\star I(Q)^{n} =&  (p_2x_0-p_0x_2,x_1)^m \st (q_1x_0-q_0x_1,x_2)^n\\
	=&\left( (p_2y_0-p_0y_2,y_1)^m + (q_1z_0-q_0z_1,z_2)^n + \mathcal{H}
	\right) \bigcap \mathbb{K}[\mathbf{x}]\\
	=& (x_1^m,x_2^n)\not\subset (x_1,x_2)^{n+m-1}= I(P\st Q)^{n+m-1}.
\end{flalign*}
Since $(x_1,x_2)^{n+m-1}\subset (x_1^m,x_2^n)\subset (x_1,x_2)^{m} $  therefore $ I(P\st Q)^{n+m-1}\subset I(P)^m\star I(Q)^{n}\subset I(P\st Q)^{m}$.

Now again without loss of generality assume that $ P=[p_0:p_1:0] $  and $ Q=[q_0:q_1:0] $. One can see that  $ x_2^{m}\in 	I(P)^m\star I(Q)^{n}$. Since $ x_2^{m}\not\in I(P\st Q)^{m+n-1} $ therefore the equality fails. Similarly one can show that $ I(P\st Q)^{n+m-1}\subset I(P)^m\star I(Q)^{n}\subset I(P\st Q)^{m}$.

\end{proof}

%%%%%%%%%%%%%%%%%%%%%%%%%%%%%%%%%%%%%%%%%%%%%%%%%%%%%%%%%%%%%%

\section{Hadamard Fat Grids} \label{HFG}

In this section we introduce and study a particular set of fat points in $\mathbb P^2$, that we call {\em Hadamard fat grid}, whose support is a complete intersection. In particular, in Theorem \ref{betti numbers HFG in P2} we describe a graded minimal free resolution of a Hadamard fat grid  using the results from Section \ref{p1xp1}, enlarging the known literature on the minimal graded resolution of homogeneous sets  of fat points in $\mathbb P^2$ supported on a complete intersection.  We also compute the Waldschmidt constant and the resurgence of the ideal defining a Hadamard fat grid (see Proposition \ref{WC} and Corollary \ref{brian}).

 Let $ P_M=\{P_1,\ldots,P_r\} $ and $ Q_N= \{Q_1,\ldots,Q_s\} $  be two sets of collinear points in $ \PP^2\setminus \Delta_1 $ with assigned positive multiplicities, respectively, $ M=\{m_1,\ldots,m_r\}$ and $ N=\{n_1,\ldots,n_s\}$.
 In terms of ideals we have
 \[
 I(P_M)=I(P_1)^{m_1}\cap \cdots \cap I(P_r)^{m_r}
 \mbox{ and }
 I(Q_N)=I(Q_1)^{n_1}\cap \cdots \cap I(Q_s)^{n_s}.\]
 
 \begin{definition}\label{defHFG} Assume that $P_i\star Q_j \not= P_k\star Q_l$ for all $1\leq i<k\leq r$ and $1\leq j<l\leq s$. Then the  set of fat points  defined by $I(P_M)\star I(Q_N)$, is called a Hadamard fat grid and it is denoted by $HFG(P_M,Q_N)$.
  \end{definition}
 
 According to Theorem \ref{thmtwopoints} and Corollary \ref{coldechad}, the ideal of $HFG(P_M,Q_N)$ is 
 \[
\bigcap_{i\in[r]}\bigcap_{j\in[s]} I(P_i\star Q_j)^{m_i+n_j-1}.
\]
 
 By Lemma 3.1 in \cite{BCK},  we know that the Hadamard product $Z\star S$ of a collinear set $Z$ by a point $S\in \PP^2\setminus \Delta_1$ is still a collinear set lying on the line $S\star L$, where $Z\subset L$. Hence, if we denote by $\ell_P$ and $\ell_Q$   the lines in which the sets $P_M$ and $Q_N$ lie respectively, one has
 \[
 P_i\star Q_j \in \ell_P\star Q_j  \,\, \mbox{ for all } i=1, \dots, r  \mbox{ and for all } j=1, \dots, s.
 \]
 And similarly,
\[
 P_i\star Q_j \in P_i\star \ell_Q  \,\, \mbox{ for all } j=1, \dots, s  \mbox{ and for all } i=1, \dots, r.
 \]
This  shows that $HFG(P_M,Q_N)$ has the structure of a planar grid. Specifically, it is a set of fat points whose support is a complete intersection of type $(r,s)$ in $\mathbb P^2$.

\begin{example}
Figure \ref{hfg1} shows a Hadamard fat grid for $r=4$ and $s=5$; Figure \ref{hfg1}(a) shows the geometric structure, with all lines involved in the grid, while in Figure \ref{hfg1}(b) we represent the multiplicities of each point of the grid.
\end{example}

\begin{figure}
\centering
\begin{tabular}{cc}
\begin{tikzpicture}
   [scale=0.55,auto=center]
   
 %The sets of points
 %Q
 \draw(9,7.3) node[scale=0.6] {$\ell_Q$};     
\draw[dotted] (-1,7) -- (-0.5,7);
\draw(-0.5,7) -- (9,7);
\draw[dotted] (9,7) -- (9.5,7);

  \node[circle, draw=black!100,fill=black!100,inner sep=2pt]  at (0,7) {};
   \node[circle, draw=black!100,fill=black!100,inner sep=2pt]  at (2,7) {};
   \node[circle, draw=black!100,fill=black!100,inner sep=2pt]  at (4,7) {};
   \node[circle, draw=black!100,fill=black!100,inner sep=2pt]  at (6,7) {};
   \node[circle, draw=black!100,fill=black!100,inner sep=2pt]  at (8,7) {};

   \draw(0,7.5) node[scale=0.7] {$Q_1$};  
   \draw(2,7.5) node[scale=0.7] {$Q_2$};  
   \draw(4,7.5) node[scale=0.7] {$Q_3$};  
   \draw(6,7.5) node[scale=0.7] {$Q_4$};  
   \draw(8,7.5) node[scale=0.7] {$Q_5$};  

%P 
 \draw(-1.7,-1.8) node[scale=0.6] {$\ell_P$};  
\draw[dotted] (-2,5.5) -- (-2,6);
\draw (-2,-2) -- (-2,5.5);
\draw[dotted] (-2,-2.5) -- (-2,-2);
  
    \node[circle, draw=black!100,fill=black!100,inner sep=2pt]  at (-2,5) {};
    \node[circle, draw=black!100,fill=black!100,inner sep=2pt]  at (-2,3) {};
    \node[circle, draw=black!100,fill=black!100,inner sep=2pt]  at (-2,1) {};
    \node[circle, draw=black!100,fill=black!100,inner sep=2pt]  at (-2,-1) {};

   \draw(-2.5,5.05) node[scale=0.7] {$P_1$};  
   \draw(-2.5,3.05) node[scale=0.7] {$P_2$};  
   \draw(-2.5,1.05) node[scale=0.7] {$P_3$};  
   \draw(-2.5,-0.95) node[scale=0.7] {$P_4$};  

% Horizontal lines   
\draw(9.3,4.7) node[scale=0.6] {$\ell_Q\star P_1$};     
\draw[dotted] (-1,5) -- (-0.5,5);
\draw(-0.5,5) -- (9,5);
\draw[dotted] (9,5) -- (9.5,5);

\draw(9.3,2.7) node[scale=0.6] {$\ell_Q\star P_2$};     
\draw[dotted] (-1,3) -- (-0.5,3);
\draw(-0.5,3) -- (9,3);
\draw[dotted] (9,3) -- (9.5,3);

\draw(9.3,0.7) node[scale=0.6] {$\ell_Q\star P_3$};     
\draw[dotted] (-1,1) -- (-0.5,1);
\draw(-0.5,1) -- (9,1);
\draw[dotted] (9,1) -- (9.5,1);

\draw(9.3,-1.3) node[scale=0.6] {$\ell_Q\star P_4$};     
\draw[dotted] (-1,-1) -- (-0.5,-1);
\draw(-0.5,-1) -- (9,-1);
\draw[dotted] (9,-1) -- (9.5,-1);

% Vertical lines   
\draw(0.7,-1.8) node[scale=0.6] {$\ell_P\star Q_1$};  
\draw[dotted] (0,5.5) -- (0,6);
\draw (0,-2) -- (0,5.5);
\draw[dotted] (0,-2.5) -- (0,-2);

\draw(2.7,-1.8) node[scale=0.6] {$\ell_P\star Q_2$};  
\draw[dotted] (2,5.5) -- (2,6);
\draw (2,-2) -- (2,5.5);
\draw[dotted] (2,-2.5) -- (2,-2);

\draw(4.7,-1.8) node[scale=0.6] {$\ell_P\star Q_3$};  
\draw[dotted] (4,5.5) -- (4,6);
\draw (4,-2) -- (4,5.5);
\draw[dotted] (4,-2.5) -- (4,-2);

\draw(6.7,-1.8) node[scale=0.6] {$\ell_P\star Q_4$};  
\draw[dotted] (6,5.5) -- (6,6);
\draw (6,-2) -- (6,5.5);
\draw[dotted] (6,-2.5) -- (6,-2);

\draw(8.7,-1.8) node[scale=0.6] {$\ell_P\star Q_5$};  
\draw[dotted] (8,5.5) -- (8,6);
\draw (8,-2) -- (8,5.5);
\draw[dotted] (8,-2.5) -- (8,-2);

%Points on \ellQ\star P1
  \node[circle, draw=black!100,fill=black!100,inner sep=2pt]  at (0,5) {};
   \node[circle, draw=black!100,fill=black!100,inner sep=2pt]  at (2,5) {};
   \node[circle, draw=black!100,fill=black!100,inner sep=2pt]  at (4,5) {};
   \node[circle, draw=black!100,fill=black!100,inner sep=2pt]  at (6,5) {};
   \node[circle, draw=black!100,fill=black!100,inner sep=2pt]  at (8,5) {};

%Points on \ellQ\star P2
  \node[circle, draw=black!100,fill=black!100,inner sep=2pt]  at (0,3) {};
   \node[circle, draw=black!100,fill=black!100,inner sep=2pt] (n1) at (2,3) {};
   \node[circle, draw=black!100,fill=black!100,inner sep=2pt] (n1) at (4,3) {};
   \node[circle, draw=black!100,fill=black!100,inner sep=2pt] (n1) at (6,3) {};
   \node[circle, draw=black!100,fill=black!100,inner sep=2pt] (n1) at (8,3) {};

%Points on \ellQ\star P3
  \node[circle, draw=black!100,fill=black!100,inner sep=2pt] (n1) at (0,1) {};
   \node[circle, draw=black!100,fill=black!100,inner sep=2pt] (n1) at (2,1) {};
   \node[circle, draw=black!100,fill=black!100,inner sep=2pt] (n1) at (4,1) {};
   \node[circle, draw=black!100,fill=black!100,inner sep=2pt] (n1) at (6,1) {};
   \node[circle, draw=black!100,fill=black!100,inner sep=2pt] (n1) at (8,1) {};

%Points on \ellQ\star P4
  \node[circle, draw=black!100,fill=black!100,inner sep=2pt] (n1) at (0,-1) {};
   \node[circle, draw=black!100,fill=black!100,inner sep=2pt] (n1) at (2,-1) {};
   \node[circle, draw=black!100,fill=black!100,inner sep=2pt] (n1) at (4,-1) {};
   \node[circle, draw=black!100,fill=black!100,inner sep=2pt] (n1) at (6,-1) {};
   \node[circle, draw=black!100,fill=black!100,inner sep=2pt] (n1) at (8,-1) {};

%Name of points in HFG

\draw(0.7,5.3) node[scale=0.5] {$Q_1\star P_1$};  
\draw(2.7,5.3) node[scale=0.5] {$Q_2 \star P_1$};  
\draw(4.7,5.3) node[scale=0.5] {$Q_3 \star P_1$};  
\draw(6.7,5.3) node[scale=0.5] {$Q_4 \star P_1$};  
\draw(8.7,5.3) node[scale=0.5] {$Q_5 \star P_1$};  

\draw(0.7,3.3) node[scale=0.5] {$Q_1 \star P_2$};  
\draw(2.7,3.3) node[scale=0.5] {$Q_2 \star P_2$};  
\draw(4.7,3.3) node[scale=0.5] {$Q_3 \star P_2$}; 
\draw(6.7,3.3) node[scale=0.5] {$Q_4 \star P_2$}; 
\draw(8.7,3.3) node[scale=0.5] {$Q_5 \star P_2$}; 

\draw(0.7,1.3) node[scale=0.5] {$Q_1 \star P_3$};  
\draw(2.7,1.3) node[scale=0.5] {$Q_2 \star P_3$};  
\draw(4.7,1.3) node[scale=0.5] {$Q_3 \star P_3$}; 
\draw(6.7,1.3) node[scale=0.5] {$Q_4 \star P_3$}; 
\draw(8.7,1.3) node[scale=0.5] {$Q_5 \star P_3$}; 

\draw(0.7,-0.7) node[scale=0.5] {$Q_1 \star P_4$};  
\draw(2.7,-0.7) node[scale=0.5] {$Q_2 \star P_4$};  
\draw(4.7,-0.7) node[scale=0.5] {$Q_3 \star P_4$}; 
\draw(6.7,-0.7) node[scale=0.5] {$Q_4 \star P_4$}; 
\draw(8.7,-0.7) node[scale=0.5] {$Q_5 \star P_4$}; 

\end{tikzpicture}&

\begin{tikzpicture}
   [scale=0.55,auto=center]
   
 %The sets of points
 %Q
 
\draw[dotted] (-1,7) -- (-0.5,7);
\draw(-0.5,7) -- (9,7);
\draw[dotted] (9,7) -- (9.5,7);

  \node[circle, draw=black!100,fill=black!100,inner sep=2pt]  at (0,7) {};
   \node[circle, draw=black!100,fill=black!100,inner sep=2pt]  at (2,7) {};
   \node[circle, draw=black!100,fill=black!100,inner sep=2pt]  at (4,7) {};
   \node[circle, draw=black!100,fill=black!100,inner sep=2pt]  at (6,7) {};
   \node[circle, draw=black!100,fill=black!100,inner sep=2pt]  at (8,7) {};

   \draw(0,7.5) node[scale=0.7] {$n_1$};  
   \draw(2,7.5) node[scale=0.7] {$n_2$};  
   \draw(4,7.5) node[scale=0.7] {$n_3$};  
   \draw(6,7.5) node[scale=0.7] {$n_4$};  
   \draw(8,7.5) node[scale=0.7] {$n_5$};  

%P 

\draw[dotted] (-2,5.5) -- (-2,6);
\draw (-2,-2) -- (-2,5.5);
\draw[dotted] (-2,-2.5) -- (-2,-2);
  
    \node[circle, draw=black!100,fill=black!100,inner sep=2pt]  at (-2,5) {};
    \node[circle, draw=black!100,fill=black!100,inner sep=2pt]  at (-2,3) {};
    \node[circle, draw=black!100,fill=black!100,inner sep=2pt]  at (-2,1) {};
    \node[circle, draw=black!100,fill=black!100,inner sep=2pt]  at (-2,-1) {};

   \draw(-2.5,5.05) node[scale=0.7] {$m_1$};  
   \draw(-2.5,3.05) node[scale=0.7] {$m_2$};  
   \draw(-2.5,1.05) node[scale=0.7] {$m_3$};  
   \draw(-2.5,-0.95) node[scale=0.7] {$m_4$};  

% Horizontal lines   
\draw[dotted] (-1,5) -- (-0.5,5);
\draw(-0.5,5) -- (9,5);
\draw[dotted] (9,5) -- (9.5,5);

\draw[dotted] (-1,3) -- (-0.5,3);
\draw(-0.5,3) -- (9,3);
\draw[dotted] (9,3) -- (9.5,3);

\draw[dotted] (-1,1) -- (-0.5,1);
\draw(-0.5,1) -- (9,1);
\draw[dotted] (9,1) -- (9.5,1);

\draw[dotted] (-1,-1) -- (-0.5,-1);
\draw(-0.5,-1) -- (9,-1);
\draw[dotted] (9,-1) -- (9.5,-1);

% Vertical lines   
\draw[dotted] (0,5.5) -- (0,6);
\draw (0,-2) -- (0,5.5);
\draw[dotted] (0,-2.5) -- (0,-2);

\draw[dotted] (2,5.5) -- (2,6);
\draw (2,-2) -- (2,5.5);
\draw[dotted] (2,-2.5) -- (2,-2);

\draw[dotted] (4,5.5) -- (4,6);
\draw (4,-2) -- (4,5.5);
\draw[dotted] (4,-2.5) -- (4,-2);

\draw[dotted] (6,5.5) -- (6,6);
\draw (6,-2) -- (6,5.5);
\draw[dotted] (6,-2.5) -- (6,-2);

\draw[dotted] (8,5.5) -- (8,6);
\draw (8,-2) -- (8,5.5);
\draw[dotted] (8,-2.5) -- (8,-2);

%Points on \ellQ\star P1
  \node[circle, draw=black!100,fill=black!100,inner sep=2pt]  at (0,5) {};
   \node[circle, draw=black!100,fill=black!100,inner sep=2pt]  at (2,5) {};
   \node[circle, draw=black!100,fill=black!100,inner sep=2pt]  at (4,5) {};
   \node[circle, draw=black!100,fill=black!100,inner sep=2pt]  at (6,5) {};
   \node[circle, draw=black!100,fill=black!100,inner sep=2pt]  at (8,5) {};

%Points on \ellQ\star P2
  \node[circle, draw=black!100,fill=black!100,inner sep=2pt]  at (0,3) {};
   \node[circle, draw=black!100,fill=black!100,inner sep=2pt] (n1) at (2,3) {};
   \node[circle, draw=black!100,fill=black!100,inner sep=2pt] (n1) at (4,3) {};
   \node[circle, draw=black!100,fill=black!100,inner sep=2pt] (n1) at (6,3) {};
   \node[circle, draw=black!100,fill=black!100,inner sep=2pt] (n1) at (8,3) {};

%Points on \ellQ\star P3
  \node[circle, draw=black!100,fill=black!100,inner sep=2pt] (n1) at (0,1) {};
   \node[circle, draw=black!100,fill=black!100,inner sep=2pt] (n1) at (2,1) {};
   \node[circle, draw=black!100,fill=black!100,inner sep=2pt] (n1) at (4,1) {};
   \node[circle, draw=black!100,fill=black!100,inner sep=2pt] (n1) at (6,1) {};
   \node[circle, draw=black!100,fill=black!100,inner sep=2pt] (n1) at (8,1) {};

%Points on \ellQ\star P4
  \node[circle, draw=black!100,fill=black!100,inner sep=2pt] (n1) at (0,-1) {};
   \node[circle, draw=black!100,fill=black!100,inner sep=2pt] (n1) at (2,-1) {};
   \node[circle, draw=black!100,fill=black!100,inner sep=2pt] (n1) at (4,-1) {};
   \node[circle, draw=black!100,fill=black!100,inner sep=2pt] (n1) at (6,-1) {};
   \node[circle, draw=black!100,fill=black!100,inner sep=2pt] (n1) at (8,-1) {};

%Multiplicities in HFG

\draw(0.85,5.3) node[scale=0.4] {$m_1+n_1-1$};  
\draw(2.85,5.3) node[scale=0.4] {$m_1+n_2-1$};  
\draw(4.85,5.3) node[scale=0.4] {$m_1+n_3-1$};  
\draw(6.85,5.3) node[scale=0.4] {$m_1+n_4-1$};  
\draw(8.85,5.3) node[scale=0.4] {$m_1+n_5-1$};  

\draw(0.85,3.3) node[scale=0.4] {$m_2+n_1-1$};  
\draw(2.85,3.3) node[scale=0.4] {$m_2+n_2-1$};  
\draw(4.85,3.3) node[scale=0.4] {$m_2+n_3-1$}; 
\draw(6.85,3.3) node[scale=0.4] {$m_2+n_4-1$}; 
\draw(8.85,3.3) node[scale=0.4] {$m_2+n_5-1$}; 

\draw(0.85,1.3) node[scale=0.4] {$m_3+n_1-1$};  
\draw(2.85,1.3) node[scale=0.4] {$m_3+n_2-1$};  
\draw(4.85,1.3) node[scale=0.4] {$m_3+n_3-1$}; 
\draw(6.85,1.3) node[scale=0.4] {$m_3+n_4-1$}; 
\draw(8.85,1.3) node[scale=0.4] {$m_3+n_5-1$}; 

\draw(0.85,-0.7) node[scale=0.4] {$m_4+n_1-1$};  
\draw(2.85,-0.7) node[scale=0.4] {$m_4+n_2-1$};  
\draw(4.85,-0.7) node[scale=0.4] {$m_4+n_3-1$}; 
\draw(6.85,-0.7) node[scale=0.4] {$m_4+n_4-1$}; 
\draw(8.85,-0.7) node[scale=0.4] {$m_4+n_5-1$}; 

\end{tikzpicture}\\
(a) & (b)
\end{tabular}
\caption{$HFG(P_M,Q_N)$.}\label{hfg1}
\end{figure}
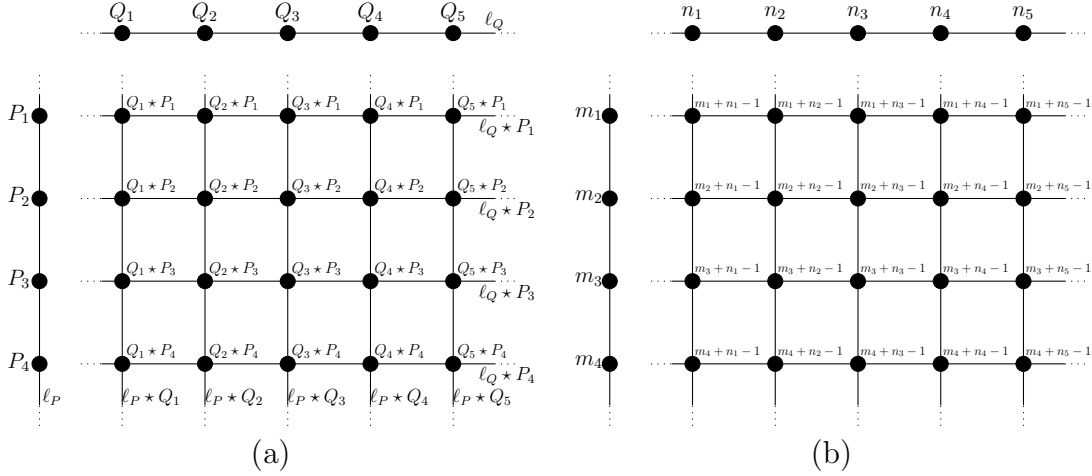

\begin{remark}\label{remarkmultiplicities}
From Figure \ref{hfg1}(b) we see that the multiplicities of the points in the grid have a specific behaviour. 
If we assume that  $m_i\leq m_{i+1}$, for $i=1, \dots, r-1$ and $ n_j\leq n_{j+1}$ for each $j=1, \dots, s-1$, then the multiplicities of $P_i\star Q_j$ and $P_i\star Q_{j+1}$ differ of $n_{j+1}-n_j$ for all $i=1, \dots, r$. Similarly  the multiplicities of $P_i\star Q_j$ and $P_{i+1}\star Q_{j}$ differ of $m_{i+1}-m_i$ for all $j=1, \dots, s$. 
Hence the Hadamard fat grids are a subclass of all possible fat grids in the plane.
\end{remark}

From the rest of the paper we assume that $s\geq r$ and the multiplicities are ordered in non-decreasing order, that is  $m_i\leq m_{i+1}$, for $i=1, \dots, r-1$ and $ n_j\leq n_{j+1}$ for each $j=1, \dots, s-1$. We denote by $\mathcal{I}(P_M,Q_N)$ the ideal of $HFG(P_M,Q_N)$.

To short the notation, set $m_{ij}=m_i+n_j-1$ for $i=1,\dots,r$ and $j=1,\dots,s$.

\begin{lemma}\label{ZY}
Let  $Z$ be a set of fat points in $\mathbb P^1\times \mathbb P^1$ whose support is a complete intersection of type $(r,s)$ (or $(r,s)$-grid))  and whose multiplicities $m_{ij}$ are the same as a Hadamard Fat grid $HFG(P_M,Q_N)$. Then $Z$ share the graded same Betti numbers as a set of fat points $Y$ in $\mathbb{P}^2$.
\end{lemma}
\begin{proof}
From Lemma \ref{acm}, $Z$ is an ACM set of fat points in $\mathbb P^1\times \mathbb P^1$  and from Remark \ref{FGM}, its ideal $I_{Z}$ defines a set of fat lines in $\mathbb P^{3}$,  and it still continues to be ACM and $1$-dimensional. Hence, after a ``proper hyperplane section" we get a set of fat points $Y$ in $\mathbb P^2$ that has the same graded Betti numbers as $Z$. 
That is, if $\ell$ is a proper hyperplane section, we have
$$R/I_{Z}\cong {\frac{R+(\ell)}{(\ell)}}/{\frac{I_Z+(\ell)}{(\ell)}}\cong K[\mathbb{P}^2]/I_Y$$
\noindent the former of which is ACM.
\end{proof}

\begin{lemma}\label{same betti in p2} Let $X$ and $X'$ be two sets of fat points in $\mathbb{P}^2$ whose support is an $(r,s)-grid$ and with the same multiplicities $m_{ij}$  as a $HFG(P_M,Q_N)$. Then $X$ and $X'$ share the same graded Betti numbers. 
\end{lemma}
\begin{proof}  Following again the same method as Theorem 3.2 in  \cite{FGM}, Remark \ref{FGM} and Lemma \ref{ZY}, we  construct  the polynomial ring $T$ in the $r+s$ new variables and form a height $2$ monomial ideal $J$ in $T$ given by the intersection of ideals of the form $(a_{1,i}, a_{2,j})^{m_{ij}}$ corresponding to the components of $X$. Thus, starting from the scheme defined by the ideal $J$, after suitable sequences of proper hyperplane sections we can construct a set (fat) of lines in $\mathbb{P}^3$  that corresponds to an ACM set of fat points $Z$ in $\mathbb P^1\times \mathbb P^1$ and since,  from Proposition \ref{acm}, $Z$ is ACM then, from Theorem 3.2 in \cite{FGM}, $T/J$ is CM . 
Analogously, we can construct a set of lines in $\mathbb{P}^3$  that corresponds to an ACM set of fat points $Z'$ in $\mathbb P^1\times \mathbb P^1$ and since $Z'$ is ACM then $T/J$ is CM. 
Again, starting from the scheme defined by the ideal $J$, after two other suitable sequences of hyperplane sections, we get two sets of fat points $X$ and $X'$  in $\mathbb P^2$ that both share the same graded Betti numbers as $Z$ and $Z'$ ( and as $J$).
\end{proof}

We show the main two results of this section.

\begin{theorem}\label{samebetti} Let $X$ be a Hadamard fat grid $HFG(P_M,Q_N)$  in $\mathbb P^2$ and $Z$ be an ACM set of fat points in $\mathbb P^1\times \mathbb P^1$   supported on an $(r,s)$-grid with the same multiplicities $m_{ij}$ as the Hadamard fat grid $X$. Then $X$ and $Z$ share the same Betti numbers.
\end{theorem}
\begin{proof} From  Lemma \ref{ZY} we can construct a set of fat points $Y$ an $(r,s)$-grid  in $\mathbb P^2$ that preserves the same multiplicities on $m_{ij}$ and the same Betti numbers as $Z$. From Lemma \ref{same betti in p2}, $X$ and $Y$ (and  $Z$) share the graded same Betti numbers.

\end{proof}

We now are able  to compute a minimal free graded resolution of a Hadamard fat grid. We will use the results from previous Section \ref{p1xp1}, to prove the main result of this section.

\begin{theorem}\label{betti numbers HFG in P2}
Let $X=HFG(P_M,Q_N)$  be a Hadamard fat grid  in $\mathbb P^2$. Then a graded minimal free
resolution of $I(X)$  is given by
\[
0 \longrightarrow \bigoplus_{(v_1,v_2) \in V_{X}} R(-v_1-v_2)
\longrightarrow \bigoplus_{(c_1,c_2) \in C_{X}} R(-c_1-c_2)
\longrightarrow I(X) \longrightarrow 0.\]
\end{theorem}
\begin{proof} We can construct $\mathcal A_X = (\alpha_1,\ldots,\alpha_m)$ associated to $X$ using the  method described as Section \ref{p1xp1}.  Applying Theorem \ref{bettinumbersftpts} and Theorem \ref{samebetti}, we get the conclusion.
\end{proof}

In the sequel we adopt the following notation:

\begin{notation} \label{notation}  Let $\mathcal{I}(P_M,Q_N)$ be the ideal of a Hadamard fat grid $HFG(P_M,Q_N)$ where $ M=\{m_1,\ldots,m_r\}$ and $ N=\{n_1,\ldots,n_s\}$ with $r\leq s$,  $m_i\leq m_{i+1}$, for $i=1, \dots, r-1$ and $ n_j\leq n_{j+1}$ for each $j=1, \dots, s-1$. 
Set $a_i=m_{r-i+1}+n_s-1$ and  $b_j=n_{s-j+1}-n_s$ for $i=1,\ldots, r$, $j=1,\ldots,s$. Let $H_i$ denotes the horizontal lines defining $\ell_Q\star P_{r-i+1}$, and $V_j$ denotes the vertical lines defining $\ell_P\star Q_{s-j+1}$. 
\end{notation}

From Corollary 3.4 in \cite{FGM} and Theorem \ref{betti numbers HFG in P2}, we have that 
\begin{corollary} \label{lineargens} If $X$ is a Hadamard fat grid $X=HFG(P_M,Q_N)$  in $\mathbb P^2$, then its homogeneous ideal is minimally generated by products of linear forms of type $H_i$ and $V_j$.
\end{corollary}

\begin{theorem}\label{generators} A minimal set of generators of the ideal $ \mathcal{I}(P_M,Q_N) $ consists of $ m_r+n_s $ generators of types $ H_1^{a_1 -k}\cdots H_r^{a_r-k}\cdot V_1^{b_1+k}\cdots V_s^{b_s+k} $
for $k=0, \dots, m_r+n_s-1$ where we adopt the convention that $H_i^{a_i-k}=1$ if $a_i-k\leq 0$ and $V_j^{b_j+k}=1$ if $b_j+k\leq 0$.  That is,  a minimal set of generators is of type
\begin{equation}
\begin{aligned}
&H_1^{m_r+n_s-1}H_2^{m_{r-1}+n_s-1}\cdots H_r^{m_1+n_s-1}\cdot V_1^{0}V_{2}^{n_{s-1}-n_s} \cdots V_{s-1}^{n_{2}-n_s} V_s^{n_1-n_s},\\
&H_1^{m_r+n_s-2}H_2^{m_{r-1}+n_s-2}\cdots H_r^{m_1+n_s-2}\cdot V_1^{1}V_{2}^{n_{s-1}-n_s+1} \cdots V_{s-1}^{n_{2}-n_s+1} V_s^{n_1-n_s+1},\\
&\vdots\\
&H_1^{0}H_2^{m_{r-1}-m_r}\cdots H_r^{m_1-m_r}\cdot V_1^{m_r+n_s-1}\cdots V_{s-1}^{n_{2}+m_r-1} V_s^{n_1+m_r-1}.
\end{aligned}
\end{equation}

\end{theorem}
\begin{proof} We apply Theorem \ref{bettinumbersftpts}, Theorem \ref{samebetti},  Theorem \ref{betti numbers HFG in P2} and Corollary \ref{lineargens} .
\end{proof}

We will show the above results with an example.

\begin{example}\label{ExCase}  Let $X=HFG(P_M,Q_N)$ be a Hadamard fat grid where $M=(2,3,3)$ and $N=(2,3,4,4)$ as in Figure \ref{example}.

\begin{figure}

\centering

\begin{tikzpicture}
   [scale=0.55,auto=center]
   
 %The sets of points
 %Q
   
\draw[dotted] (-1,7) -- (-0.5,7);
\draw(-0.5,7) -- (9,7);
\draw[dotted] (9,7) -- (9.5,7);

  \node[circle, draw=black!100,fill=black!100,inner sep=2pt]  at (0,7) {};
   \node[circle, draw=black!100,fill=black!100,inner sep=2pt]  at (2,7) {};
   \node[circle, draw=black!100,fill=black!100,inner sep=2pt]  at (4,7) {};
   \node[circle, draw=black!100,fill=black!100,inner sep=2pt]  at (6,7) {};

   \draw(0,7.5) node[scale=0.7] {$2$};  
   \draw(2,7.5) node[scale=0.7] {$3$};  
   \draw(4,7.5) node[scale=0.7] {$4$};  
   \draw(6,7.5) node[scale=0.7] {$4$};

%P 

\draw (-2,-2) -- (-2,5.5);

    \node[circle, draw=black!100,fill=black!100,inner sep=2pt]  at (-2,5) {};
    \node[circle, draw=black!100,fill=black!100,inner sep=2pt]  at (-2,3) {};
    \node[circle, draw=black!100,fill=black!100,inner sep=2pt]  at (-2,1) {};

   \draw(-2.5,5.05) node[scale=0.7] {$2$};  
   \draw(-2.5,3.05) node[scale=0.7] {$3$};  
   \draw(-2.5,1.05) node[scale=0.7] {$3$};

% Horizontal lines   

\draw(-0.5,5) -- (9,5);
\draw(-0.5,3) -- (9,3);
\draw[dotted] (9,3) -- (9.5,3);
\draw(-0.5,1) -- (9,1);

% vertical  lines   
\draw (0,-2) -- (0,5.5);
\draw (2,-2) -- (2,5.5);
\draw (4,-2) -- (4,5.5);
\draw (6,-2) -- (6,5.5);

%Points on \ellQ\star P1
  \node[circle, draw=black!100,fill=black!100,inner sep=2pt]  at (0,5) {};
   \node[circle, draw=black!100,fill=black!100,inner sep=2pt]  at (2,5) {};
   \node[circle, draw=black!100,fill=black!100,inner sep=2pt]  at (4,5) {};
   \node[circle, draw=black!100,fill=black!100,inner sep=2pt]  at (6,5) {};
%Points on \ellQ\star P2
  \node[circle, draw=black!100,fill=black!100,inner sep=2pt]  at (0,3) {};
   \node[circle, draw=black!100,fill=black!100,inner sep=2pt] (n1) at (2,3) {};
   \node[circle, draw=black!100,fill=black!100,inner sep=2pt] (n1) at (4,3) {};
   \node[circle, draw=black!100,fill=black!100,inner sep=2pt] (n1) at (6,3) {};

%Points on \ellQ\star P3
  \node[circle, draw=black!100,fill=black!100,inner sep=2pt] (n1) at (0,1) {};
   \node[circle, draw=black!100,fill=black!100,inner sep=2pt] (n1) at (2,1) {};
   \node[circle, draw=black!100,fill=black!100,inner sep=2pt] (n1) at (4,1) {};
   \node[circle, draw=black!100,fill=black!100,inner sep=2pt] (n1) at (6,1) {};
%Name of points in HFG

\draw(0.7,5.3) node[scale=0.5] {$3$};  
\draw(2.7,5.3) node[scale=0.5] {$4$};  
\draw(4.7,5.3) node[scale=0.5] {$5$};  
\draw(6.7,5.3) node[scale=0.5] {$5$};

\draw(0.7,3.3) node[scale=0.5] {$4$};  
\draw(2.7,3.3) node[scale=0.5] {$5$};  
\draw(4.7,3.3) node[scale=0.5] {$6$}; 
\draw(6.7,3.3) node[scale=0.5] {$6$};

\draw(0.7,1.3) node[scale=0.5] {$4$};  
\draw(2.7,1.3) node[scale=0.5] {$5$};  
\draw(4.7,1.3) node[scale=0.5] {$6$}; 
\draw(6.7,1.3) node[scale=0.5] {$6$};

\end{tikzpicture}
\caption{$X=HFG(P_M,Q_N)$, with $M=(2,3,3)$ and $N=(2,3,4,4)$.}
\label{example}
\end{figure}
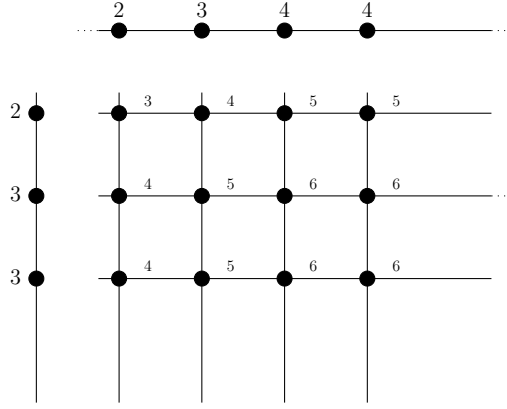

In this case,
$$\mathcal A_X=(21,21,17,17,17,13,13,13,9,9,9,5,5,5,2,2,2),$$
$$V_X=\{(2,21),(5,17),(8,13),(11,9),(14,5), (17,2)\}$$ and $$C_X=\{(0,21),(2,17),(5,13),(8,9),(11,5), (14,2), (17,0)\} .$$ Then, using Theorem \ref{betti numbers HFG in P2}, a graded minimal free
resolution of $I(X)$  is given by
\[
0 \longrightarrow   R(-23)\oplus R(-22)\oplus R(-21)\oplus R(-20)\oplus R^2(-19)
\longrightarrow \]
\[ R(-21)\oplus R(-19)\oplus R(-18)\oplus R^2(-17)\oplus R^2(-16)
\longrightarrow I(X) \longrightarrow 0\]

And using Theorem \ref{generators}, a minimal set of seven generators is given by
\begin{equation}
\begin{aligned}
&H_1^{6}H_2^{6} H_3^{5}\\
&H_1^5 H_2^5 H_3^4 \cdot V_1^{1}V_2^{1}\\
&H_1^4 H_2^4 H_3^3 \cdot V_1^{2}V_2^{2}V_3^{1}\\
&H_1^3 H_2^3 H_3^2 \cdot V_1^{3}V_2^{3}V_3^{2} V_{4}^{1}\\
&H_1^2H_2^2H_3^{1}\cdot V_1^{4}V_2^{4}V_3^{3} V_{4}^{2}\\
&H_1^{1}H_2^{1}\cdot V_1^{5}V_2^{5}V_3^{4} V_{4}^{3}\\
&V_1^{6}V_2^{6}V_3^{5} V_{4}^{4}.\\
\end{aligned}
\end{equation}

\end{example}

We are able to compute the Waldschmidt constant and the resurgence of a Hadamard fat grid  $ \mathcal{I}(P_M,Q_N)$. We recall the following
\begin{definition} For a given homogeneous ideal $0\ne I\subsetneq \mathbb{K}\left[x_{0}, \ldots, x_{N}\right]$, its  Waldschmidt constant $\hat{\alpha}(I)$ is defined as 
$$\hat{\alpha}(I)=\hbox{lim}_{m\to \infty}\alpha(I^{(m)})/m$$
\noindent where $\alpha(I)$ is the least degree $d$ such that $I_d\neq (0)$.
\end{definition}
Because of the subadditivity of $\alpha$,
this limit exists (see Lemma 2.3.1 of \cite{BC}).
Moreover, $\hat{\alpha}(I)>0$ (see Lemma 2.3.2 of \cite{BC}).

We need some preliminary result. 

\begin{proposition}\label{alphabeta}
	Let $X=HFG(P_M,Q_N)$ be a Hadamard fat grid. The minimal degree $ \alpha(\mathcal{I}(P_M,Q_N))$ of a generator of the ideal $\mathcal{I}(P_M,Q_N)$ is given by
	$$ \alpha(\mathcal{I}(P_M,Q_N)) =\sum_{i=1}^{r} m_{i}+\sum_{i=1}^{r} n_{s-i+1} -r.$$
\end{proposition}
\begin{proof}

	By Theorem \ref{generators}, we know that the degrees of the minimal generators of the ideal $ \mathcal{I}(P_M,Q_N) $ are of type
	$$ \left\lbrace \sum_{i=1}^{r} \deg H_{r-i+1}^{m_i+n_s-t}  + \sum_{j=1}^{s}\deg V_{s-j+1}^{n_j-n_s+t-1}\right\rbrace _{t= 1}^{m_r+n_s },$$
	where $H_i^{a_i-k}=1$ if $a_i-k\leq 0$ and $V_j^{b_j+k}=1$ if $b_j+k\leq 0$ for $k=0, \dots, m_r+n_s-1$ . It means that we only consider non-negative degrees.

	Since $ m_1 \leq m_2 \leq \cdots \leq m_r  $, therefore $$\deg H_{r}^{m_1+n_s-t} \leq \deg H_{r-1}^{m_2+n_s-t} \leq \cdots \leq \deg H_{1}^{m_r+n_s-t}. $$
	On the other hand, $\deg H_{r}^{m_1+n_s-t} > 0$ for $ t = 1,\ldots, m_1+n_s-1 $. It follows that  
	\begin{equation}\label{sumh}
		0<\deg H_{r}^{m_1+n_s-t} \leq \deg H_{r-1}^{m_2+n_s-t} \leq \cdots \leq \deg H_{1}^{m_r+n_s-t} \quad \forall\  t = 1,\ldots, m_1+n_s-1 .
	\end{equation}
	We can have a similar argument for the other summation as follows.
	
	Since $ n_s \geq n_{s-1} \geq \cdots \geq n_{s-r+1}\geq \cdots \geq n_1 $,  therefore $$\deg V_{s}^{n_1-n_s+t-1} \leq\cdots \leq \deg V_{r}^{n_{s-r+1}-n_s+t-1} \leq \cdots\leq \deg V_{2}^{n_{s-1}-n_s+t-1} \leq \deg V_{1}^{n_s-n_s+t-1}. $$
	One can also observe that $ \deg V_{1}^{n_s-n_s+t-1} = 0 $ for $ t = 1 $. Hence for $ t = 1 $ we have,
	$$\deg V_{s}^{n_1-n_s} \leq\cdots \leq \deg V_{r}^{n_{s-r+1}-n_s} \leq \cdots\leq \deg V_{2}^{n_{s-1}-n_s} \leq \deg V_{1}^{n_s-n_s} = 0. $$

	For $ t = 1,\ldots, m_r+n_s $, let $ G = \{g_1,g_2,\ldots,g_t\} $  be a set of minimal generators of $ \mathcal{I}(P_M,Q_N) $.  We claim that $ g_{t} \geq   g_{n_s - n_{s-r+1}} $ for $ 1 \leq t \leq n_s - n_{s-r+1}  $.

	 Set $	\alpha= \sum_{i=1}^{r}m_i + (\sum_{i=1}^{r}n_{s-i+1}) - r$.

We observe that 
 \begin{align*}
	 	 \alpha&=\sum_{i=1}^{r}(m_i +n_s - n_s + n_{s-r+1}) + \sum_{j=1}^{s}(n_j-n_s+n_s-n_{s-r+1}-1)\\
			& = \sum_{i=1}^{r}(m_i +n_{s-r+1} ) + \sum_{j=1}^{s} (n_j - n_{s-r+1} -1)\\
			& = \sum_{i=1}^{r}m_i + rn_{s-r+1}\\
			& + (n_1- n_{s-r+1}-1) +(n_2- n_{s-r+1}-1)+\cdots\\
			& + (n_{s-r}- n_{s-r+1}-1) + (n_{s-r+1}- n_{s-r+1}-1)\\
			& + (n_{s-r+2}- n_{s-r+1}-1) + \cdots + (n_{s}- n_{s-r+1}-1).
	 \end{align*}
	 Since 
\begin{equation}\label{negative exponents}
 n_s \geq n_{s-1} \geq \cdots \geq n_{s-r+1} \geq n_{s-r} \geq\cdots \geq n_{1} ,
\end{equation}
\noindent we have
	\begin{align*}
	 \alpha &= \sum_{i=1}^{r}m_i + rn_{s-r+1} \\
	 & + (n_{s-r+1}- n_{s-r+1}-1) + (n_{s-r+2}- n_{s-r+1}-1) + \cdots + (n_{s}- n_{s-r+1}-1)\\
	 & = \sum_{i=1}^{r}m_i + rn_{s-r+1} + \sum_{i=1}^r n_{s-i+1} - r n_{s-r+1} -r\\
	 & = \sum_{i=1}^{r}m_i + \sum_{i=1}^r n_{s-i+1} -r.
	\end{align*}
	
We also note that from (\ref{negative exponents}), we have
\begin{equation}\label{negative sum}
(\sum_{i=r+1}^s n_{s-i+1} - n_s ) \leq 0.
\end{equation}
Then, for $ 1\leq t \leq n_s - n_{s-r+1} $ and using (\ref{negative sum}), we have
	\begin{align*}
g_t\geq &\alpha\\
	\iff & \sum_{i=1}^r\left(m_i+n_s-t\right)+\sum_{j=1}^s\left(n_j-n_s+ t-1\right) \geq \sum_{i=1}^r m_i+(\sum_{i=1}^r n_{s-i+1})-r  \\
	\iff& (\sum_{j=1}^s n_j) + r n_s-rt- sn_s+s(t-1) \geq (\sum_{i=1}^r n_{s-i+1})-r\\
	\iff& (\sum_{i=1}^r n_{s-i+1}) + (\sum_{i=r+1}^s n_{s-i+1})+ r n_s-rt- sn_s+s(t-1) \geq (\sum_{i=1}^r n_{s-i+1})-r\\
	\iff& (\sum_{i=r+1}^s n_{s-i+1})- (s-r)n_s+ s(t-1) +r -rt \geq 0\\
	\iff& (\sum_{i=r+1}^s n_{s-i+1} - n_s )+ s(t-1) +r -rt \geq 0  \\
	\iff& t(s-r) +r -s \geq 0 \\
\iff & t \geq 1 .
	\end{align*}
This proves that $\alpha$ is the minimum value among all the degrees of the minimal generators of a Hadamard fat grid.

\end{proof}

\begin{example}
Using again the values of Example \ref{ExCase} we get 
\[
 \alpha(\mathcal{I}(P_M,Q_N)) =2+3+3+4+4+3-3=16
 \]
 as expected from its minimal resolution. In this case we have two generators of minimal degree, namely
$H_1^5 H_2^5 H_3^4 \cdot V_1^{1}V_2^{1}$ and $H_1^4 H_2^4 H_3^3 \cdot V_1^{2}V_2^{2}V_3^{1}$.
\end{example}

\begin{remark}
 Note that since $ m_1 \leq m_2 \leq \cdots \leq m_r  $ and  $ n_1 \leq n_2 \leq \cdots \leq n_s  $ therefore the maximum degree of generators of the ideal $ \mathcal{I}(P_M,Q_N) $ is as the following: $$\beta(\mathcal{I}(P_M,Q_N))=\max \left\lbrace \sum_{j=1}^s (m_r+n_j-1),\sum_{i=1}^r (n_s+m_i-1)\right\rbrace .$$
\end{remark}

We need  a preliminary lemma to study the symbolic powers $\mathcal{I}(P_M,Q_N)^{(t)}$. 

\begin{lemma}\label{fathfg}
The $ t $-th symbolic power of $\mathcal{I}(P_M,Q_N)$ is a Hadamard fat grid.
\end{lemma}
\begin{proof}
	We know that:
$$
\mathcal{I}(P_M,Q_N)= \bigcap_{i=1}^r\bigcap_{j=1}^s I(P_i\st Q_j)^{m_i+n_j-1}.
$$
Hence
\begin{flalign*}
	\mathcal{I}(P_M,Q_N)^{(t)}&= \bigcap_{i=1}^r\bigcap_{j=1}^s I(P_i\st Q_j)^{t(m_i+n_j-1)}\\
	&= \bigcap_{i=1}^r\bigcap_{j=1}^s I(P_i\st Q_j)^{((tm_i-(t-1))+tn_j)-1}.
\end{flalign*}
Therefore, the sets $P_{M'} \st Q_{N'} $ with multiplicities $M'=\{tm_1-(t-1),\ldots,tm_r-(t-1)\}$ and $ N'=\{tn_1,\ldots,tn_s\}$ is the Hadamard fat grid whose ideal is $\mathcal{I}(P_M,Q_N)^{(t)}$.

\end{proof}

We are now able to show  the following

\begin{proposition}\label{WC}
The Waldschmidt constant $\hat{\alpha}(\mathcal{I}(P_M,Q_N))$ of the Hadamard fat grid $HFG(P_M,Q_N)$ is equal to the least degree of a minimal set of generators of its defining ideal, i.e. $
 \hat{\alpha}(\mathcal{I}(P_M,Q_N))= \alpha(\mathcal{I}(P_M,Q_N))$.
\end{proposition}
\begin{proof}
Since the $t$-symbolic power of the ideal of a Hadamard fat grid $HFG(P_M,Q_N)$ defines a new Hadamard fat grid $HFG(P_{M'},Q_{N'})$, one has
\[
\alpha(\mathcal{I}(P_M,Q_N)^{(t)})=\alpha(\mathcal{I}(P_{M'},Q_{N'}))
\]
where $M'$ and $N'$ are defined as in Lemma \ref{fathfg}.
Hence we can apply Proposition \ref{alphabeta} to compute  $\alpha(\mathcal{I}(P_{M'},Q_{N'}))$ obtaining

\begin{flalign*}
\alpha(\mathcal{I}(P_M,Q_N))^{(t)}) &=\sum_{i=1}^{r} \left( tm_{i}-(t-1)\right) +t \sum_{i=1}^{r} n_{s-i+1} -r\\
                           &= t\sum_{i=1}^{r} m_{i} -r(t-1)+t\sum_{i=1}^{r} n_{s-i+1}-r\\
                           &= t\sum_{i=1}^{r} m_{i} +t\sum_{i=1}^{r} n_{s-i+1}-rt\\
                           &= t\left( \sum_{i=1}^{r} m_{i} +\sum_{i=1}^{r} n_{s-i+1}-r\right) \\
                           & = t  \alpha(\mathcal{I}(P_M,Q_N)). 
\end{flalign*}
Therefore,
$$ \hat{\alpha}(\mathcal{I}(P_M,Q_N))=\lim_{t\to\infty}\dfrac{\alpha(\mathcal{I}(P_M,Q_N)^{(t)})}{t}=\alpha(\mathcal{I}(P_M,Q_N)) .$$
\end{proof}

We pass now to compute the resurgence $\rho(I)=\hbox{sup}\{m/r : I^{(m)}\not\subseteq I^r\}$ of $I=\mathcal{I}(P_M,Q_N)$. 

It is important to notice that even though $\mathcal{I}(P_M,Q_N)$ defines a set of fat points in $\mathbb P^2$ that  is not a complete intersection, we will show that $\mathcal{I}(P_M,Q_N)^t=\mathcal{I}(P_M,Q_N)^{(t)}$ and hence that its  resurgence is equal to 1. 

\begin{proposition}\label{equality symb and reg}
Let $\mathcal{I}(P_M,Q_N)$ be the ideal of a Hadamard fat grid, then $\mathcal{I}(P_M,Q_N)^t=\mathcal{I}(P_M,Q_N)^{(t)}$ for all $t\geq 1$.
\end{proposition}

\begin{proof}
By Theorem \ref{generators}, a minimal set of generators of the ideal $\mathcal{I}(P_M,Q_N) $ consists of $ m_r+n_s $ generators of types 
\begin{equation}\label{genI}
    g_k=H_1^{a_1 -k}\cdots H_r^{a_r-k}\cdot V_1^{b_1+k}\cdots V_s^{b_s+k}, 
\end{equation} 
where
\begin{equation*}
    a_i=m_{r-i+1}+n_s-1, \quad (i=1,\ldots, r), \qquad b_j=n_{s-j+1}-n_s, \quad (j=1,\ldots,s),
\end{equation*}
\begin{equation*}
    k=0,1, \dots, m_r+n_s-1.
\end{equation*}

The ordinary power $\mathcal{I}(P_M,Q_N)^t$ is generated by all possible products of $t$ generators of $\mathcal{I}(P_M,Q_N)$: 
\begin{equation}\label{tpower}
g_0^{t_0}g_1^{t_1}\cdots g_{m_r+n_s-1}^{t_{m_r+n_s-1}}
\end{equation}
where $\sum t_i=t$.
Substituting  (\ref{genI}) in (\ref{tpower}) we get that the generators of $I^t$ are of the form
\begin{equation*}\label{genI_tpower}
    H_1^{ta_1 -k}\cdots H_r^{ta_r-k}\cdot V_1^{tb_1+k}\cdots V_s^{tb_s+k}, 
\end{equation*} 
where $k=0, \dots t( m_r+n_s-1)$.

By Lemma \ref{fathfg}, we know that $\mathcal{I}(P_M,Q_N)^{(t)}$ is the ideal of a Hadamard fat grid given by the sets $P_{M'} \star Q_{N'} $ with multiplicities $M'=\{tm_1-(t-1),\ldots,tm_r-(t-1)\}$ and $ N'=\{tn_1,\ldots,tn_s\}$.
Hence, again by Theorem \ref{generators}, $\mathcal{I}(P_M,Q_N)^{(t)}$ is generated by
\begin{equation}
\label{eq:gens}
    \overline{g_k}=H_1^{\overline{a_1} -\overline{k}}\cdots H_r^{\overline{a_r}-\overline{k}}\cdot V_1^{\overline{b_1}+\overline{k}}\cdots V_s^{\overline{b_s}+\overline{k}}, 
\end{equation} 
where
\begin{equation*}
    \overline{a_i}=tm_{r-i+1}-t+1-tn_s-1=t(m_{r-i+1}+n_s-1)=ta_i, \qquad (i=1,\ldots, r),
\end{equation*}
\begin{equation*}
    \overline{b_j}=tn_{s-j+1}-tn_s=t(n_{s-j+1}-n_s)=tb_j, \qquad (j=1,\ldots,s),
\end{equation*}
and
\begin{equation*}
    \overline{k}=0,1, \dots, tm_r-t+1+tn_s-1.
\end{equation*}

Since $ tm_r-t+1+tn_s-1=t(m_r+n_s-1)$, every generator in $\mathcal{I}(P_M,Q_N)^{(t)}$ is also a generator in $\mathcal{I}(P_M,Q_N)^t$, giving $\mathcal{I}(P_M,Q_N)^{(t)}\subset \mathcal{I}(P_M,Q_N)^t$. The other inclusion  is obvious by definition of symbolic powers, hence we get $\mathcal{I}(P_M,Q_N)^t=\mathcal{I}(P_M,Q_N)^{(t)}$.
\end{proof}

Using the definition of resurgence we get the following

\begin{corollary} \label{brian}
Let $\mathcal{I}(P_M,Q_N)$ be the ideal of a Hadamard fat grid, then $\rho(\mathcal{I}(P_M,Q_N))=1$.
\end{corollary}

\begin{remark} Corollary \ref{brian} can also be recovered by the recent Theorem 2.3 in \cite{HKZ} where the authors provide various examples and questions about computing the resurgence of homogeneous ideals.
\end{remark}

\bibliographystyle{abbrv}

\begin{thebibliography}{abbrv}

\bibitem{PSC} T.\ Bauer, S.\ Di Rocco, B.\ Harbourne, M.\ Kapustka, A.\ Knutsen, W.\
Syzdek, and T.\ Szemberg. {\it A primer on Seshadri constants}, 
to appear in the AMS Contemporary Mathematics series volume 
``Interactions of Classical and Numerical Algebraic Geometry,"
Proceedings of a conference in honor of A.J. Sommese, held at Notre Dame, May 22--24 (2008).

\bibitem{BJC} I. Bahmani Jafarloo and G. Calussi, {\it Weak Hadamard star configurations and apolarity}, Rocky Mountain J. Math. {\bf 50}(3), 851--862  (2020).


\bibitem{BCC22} C. Bocci, C. Capresi and  D. Carrucoli, {\it Gorenstein points in $\PP^3$ via Hadamard products of projective varieties}, Collectanea Mathematica, https://doi.org/10.1007/s13348-022-00362-9 (2022)

\bibitem{BH} C.\ Bocci and B.\ Harbourne. {\it Comparing 
powers and symbolic powers of ideals}, Journal of Algebraic Geometry {\bf 19}, 399--417 (2010).


\bibitem{BCFL1} C. Bocci, G. Calussi, G. Fatabbi and A. Lorenzini, {\it On Hadamard product of linear varieties}, Journal of Algebra and its applications, {\bf 16}(8), 155--175 (2017).

\bibitem{BCFL2} C. Bocci, G. Calussi, G. Fatabbi and A. Lorenzini, {\it The Hilbert function of some Hadamard products}, Coll. Mathematica, {\bf 69}(2), 205--220  (2018).

\bibitem{BC} C. Bocci and E. Carlini  {\it Hadamard products of hypersurfaces}, Journal of Pure and Applied Algebra, {\bf 226}(11), (2022).

\bibitem{BCK} Bocci C., Carlini E.  and Kileel J. {\it Hadamard products of linear spaces}. Journal of Algebra, {\bf 448}, 595--617 (2016).

\bibitem{CCFG} E. Carlini, M. V. Catalisano, G. Favacchio and E. Guardo, {\it Rational normal curves and Hadamard products}, Mediterr. J. Math. {\bf 19}, 134 (2022).

\bibitem{CCGVT} E. Carlini, M. V. Catalisano, E. Guardo and A. Van Tuyl, {\it Hadamard star configurations}, Rocky Mountain J. Math. {\bf 49}(2), 419--432 (2019).

\bibitem{CooperGuardo} S.M. Cooper, E. Guardo, {\em Fat points, partial intersections and Hamming distance}, J. Algebra Appl. 19 (4), 2050071 (2020).


\bibitem{ELS} L. Ein, R. Lazarsfeld and K. Smith. {\it Uniform bounds and 
symbolic powers on smooth varieties}, Invent. Math. {\bf 144}. 241--252 (2001).

\bibitem{FGM} G. Favacchio, E. Guardo and J. Migliore, {\it On the arithmetically Cohen-Macaulay property for sets of points in multiprojective spaces}, Proceedings of the American Mathematical Society. {\bf 146}, 2811--2825, ISSN: 0002-9939, doi: 10.1090/proc/13981 (2018).

\bibitem{FM1} G. Favacchio, J. Migliore, {\it Multiprojective spaces and the arithmetically Cohen-Macaulay property}, Mathematical Proceedings of the Cambridge Philosophical Society, {\bf 166}(3), 582--597 (2019).

\bibitem{FM2} G. Favacchio, J. Migliore, {\it The ACM property for unions of lines in $\mathbb P^1\times \mathbb P^2$}, J. Pure  Applied Algebra {\bf 225}(11), doi: https://doi.org/10.1016/j.jpaa.2021.106739  (2021).

\bibitem{GHVT} E. Guardo, B. Harbourne, A. Van Tuyl, {\it Fat lines in $\mathbb P^3$: regular versus symbolic powers}, J. Algebra, {\bf 390}, 221--230, http://dx.doi.org/10.1016/j.jalgebra.2013.05.028 (2013).

\bibitem{GVT}  E. Guardo, A. Van Tuyl, {\it Arithmetically Cohen-Macaulay Sets of Points in P1 x P1}, Springerbriefs in Mathematics, p. 1-134, Springer, ISBN: 978-3-319-24164-7, ISSN: 2191-8198, doi: 10.1007/978-3-319-24166-1 (2015).
\bibitem{GVT1} E. Guardo, A. Van Tuyl, {\it Powers of complete intersections and fat points in special position}.  Illinois J. Math. 49, no. 1, 265--279 (2005).

\bibitem{HKZ} B. Harbourne, J. Kettinger and F. Zimmitti, {\it Extreme values of the resurgence for homogeneous ideals in polynomial rings}. J. Pure Applied Alg. {\bf 229}, (2022).

\bibitem{HH1}
M.\ Hochster and C.\ Huneke. {\it Comparison of symbolic and ordinary powers of
ideals}, Invent. Math. {\bf 147}2, 349--369  (2002).

\bibitem{MN} J. Migliore and U. Nagel, {\it Reduced arithmetically Gorenstein schemes and simplicial polytopes with maximal Betti numbers}, Adv. Math. {\bf 180}(1), 1--63  (2003).

\bibitem{S08} S Sullivant, {\it Combinatorial symbolic powers}. J. Algebra, {\bf 319}(1), 115-142 (2008).
	
\bibitem{SS06} B. Sturmfels and S. Sullivant, {\it Combinatorial Secant Varieties}. Pure and Applied Mathematics Quarterly, {\bf 2}(3), 867--891 (2006). 
	
\bibitem{ZS75} O. Zariski and P. Samuel,  {\it Commutative algebra} Vol. II Springer-Verlag, New York-Heidelberg (1975).




\end{thebibliography}

\end{document}